\documentclass[a4paper,fleqn]{cas-sc}

\usepackage{graphicx}%
\usepackage{multirow}%
\usepackage{amsmath,amssymb,amsfonts}%
\usepackage{amsthm}%
\usepackage[title]{appendix}%
\usepackage{xcolor}%
\usepackage{textcomp}%
\usepackage{manyfoot}%
\usepackage{booktabs}%
\usepackage{algorithm}%
\usepackage{algorithmicx}%
\usepackage{algpseudocode}%
\usepackage{listings}%
\usepackage{bigints}
\usepackage{physics}

\usepackage[numbers]{natbib}

\def\tsc#1{\csdef{#1}{\textsc{\lowercase{#1}}\xspace}}
\tsc{WGM}
\tsc{QE}

\newtheorem{theorem}{Theorem}
\newtheorem{lemma}[theorem]{Lemma}
\newtheorem{proposition}[theorem]{Proposition}
\newtheorem{note}[theorem]{Note}
\newdefinition{remark}[theorem]{Remark}
\newtheorem{example}{Example}

\begin{document}
	\let\WriteBookmarks\relax
	\def\floatpagepagefraction{1}
	\def\textpagefraction{.001}
	
	\shorttitle{A superconvergent approximation to Uryshon integral equations}    
	
	\shortauthors{S.K. Shukla and G. Rakshit}  
	
	\title [mode = title]{Acceleration of convergence in approximate solutions of Urysohn integral equations with Green's kernels}  
	
	
	
	%
	
	\author[]{Shashank K. Shukla}[orcid=0009-0008-1345-6097]

	
	\ead{shashankks@rgipt.ac.in}
	
	

	\author[]{Gobinda Rakshit}[orcid=0000-0002-5813-4656]

    \affiliation[]{organization={Department of Mathematical Sciences, Rajiv Gandhi Institute of Petroleum Technology},
		addressline={Jais}, 
		city= {Amethi},
		postcode={229304}, 
		state={Uttar Pradesh},
		country={India}}
	
	\cormark[1]
	
	
	\ead{g.rakshit@rgipt.ac.in}
	
	
	
	\cortext[1]{Corresponding author}
	
	
	
	\begin{abstract}
		Consider a non-linear operator equation $x - K(x) = f$, where $f$ is a given function and $K$ is a Urysohn integral operator with Green's function type kernel defined on $L^\infty [0, 1]$.  We apply approximation methods based on interpolatory projections onto the approximating space  $\mathcal{X}_n$, which is the space of piecewise polynomials of even degree with respect to a uniform partition of $[0, 1]$. The approximate solutions obtained from these methods demonstrate enhanced accuracy compared to the classical collocation solution for the same equation. Numerical examples are given to support our theoretical results.
	\end{abstract}
	
	

	\begin{keywords}
		Urysohn integral operator \sep Green's kernel \sep Interpolatory operator \sep Collocation points
		
	\MSC[2008] 45G10, 45L05, 65D05, 65J15, 65R20
	\end{keywords}
 
	\maketitle
	
\section{Introduction}

Let $\mathcal{X} = L^\infty[0, 1]$ and $\Omega = [0, 1] \times [0, 1] \times \mathbb{R}$. Let $\kappa : \Omega \mapsto \mathbb{R}$ be a continuous function.  Consider the non-linear integral operator
$K : L^\infty [0,1] \mapsto C[0,1]$ as
\begin{equation} \label{Eq: 01}
	K(x)(s) = \int_{0}^{1} \kappa(s, t, x(t))~dt, \quad x \in \mathcal{X}, \quad s \in [0, 1]. 
\end{equation}
Such an operator is called as Urysohn integral operator.  Since the kernel $\kappa$ is continuous, $K$ is compact operator on $\mathcal{X}$. For given $f$ and $\kappa$, we consider the following Urysohn integral equation
\begin{equation} \label{Eq: 02}
	x - K(x) = f. 
\end{equation}
The Fr\'echet derivative of $K$ at $x \in \mathcal{X}$ is given by 
\begin{equation*} 
	(K'(x)y)(s) = \int_{0}^{1} \frac{\partial \kappa}{\partial u}(s, t, x(t)) y(t) ~dt, \quad y \in \mathcal{X}, \quad s \in [0, 1],
\end{equation*}
where $\frac{\partial \kappa}{\partial u}(s, t, x(t))$ represents the partial derivative of $\kappa$ with respect to its third component $u$ at $x(t)$. Assuming that the equation \eqref{Eq: 02} possesses a solution $\varphi$ and, $1$ is not an eigenvalue of the compact linear operator $K'(\varphi)$. Then, $\varphi$ is an isolated solution of \eqref{Eq: 02}. See \cite{MAK1, MAK3}. Note that, $\varphi \in C^ \alpha [0,1]$ whenever $f \in C^ \alpha [0,1]$ for any positive integer $\alpha$. (See \cite[Corollory 3.2]{KEA1}). Solving the equation \eqref{Eq: 02} exactly, is often challenging, so we focus on finding approximate solutions. In collocation method, we employ a sequence of finite-rank interplatory projections that converge pointwise to the identity operator to approximate the operator $K$. This method has been extensively studied in \cite{MAK1, MAK3, MAK2}.

Let $\mathcal{X}_n$ denotes a sequence of finite-dimensional subspaces approximating $\mathcal{X}$, and let $\left\{ Q_n \right\}$ represents a sequence of interpolory projections from $\mathcal{X}$  to $\mathcal{X}_n$. In the collocation method, equation \eqref{Eq: 02} is approximated by
\begin{equation*} 
	\varphi_{n}^C - Q_nK(\varphi_{n}^C) = Q_nf.
\end{equation*}
For linear integral equations, Sloan \cite{Sl1} introduced the iterated collocation solution, obtained through a single iteration of the collocation solution. This approach was further opted by Atkinson-Potra \cite{KEA1} for Urysohn integral equations. The iterated collocation solution is defined by 
\begin{equation*}
	\varphi_{n}^S = K(\varphi_{n}^C) + f.
\end{equation*}
Additionally, there exists an enhanced approach known as the modified collocation method, proposed by Grammont-Kulkarni \cite{LGRPK1} as
\begin{equation*} 
	\varphi_{n}^M - K_n^M(\varphi_{n}^M) = f, \quad \mbox{ with }~  K_n^M(x) = Q_nK(x) + K(Q_nx) - Q_nK(Q_nx).
\end{equation*}
This approach generalizes the modified collocation method for the linear case, as proposed by Kulkarni \cite{RPK1}. The solution obtained by applying a single iteration in $\varphi_{n}^M$ is termed the iterated modified collocation solution which is defined by
\begin{equation*} 
	\widetilde{\varphi}_n^M=K(\varphi_n^M)+f. 
\end{equation*} 

For positive integer $r \geq 0$, let $\mathcal{X}_n$ be the space of piecewise polynomials of degree $ \le 2r$ associated with the uniform partition of $[0,1]$ defined as $\{0 = t_0 < t_1 < \cdots < t_n = 1\}$, where $t_j = \frac{j}{n}$ for $j = 1,2, \ldots, n$ with subinterval length $\displaystyle{h = \frac{1}{n}}$. In each subinterval $[t_{j-1}, t_j]$, choose $2r+1$ collocation points as Gauss points (i.e. zeros of the Legendre polynomials), and let $Q_n: \mathcal{X} \mapsto \mathcal{X}_n$ be the interpolatory projection at the collocation points. Then, the following orders of convergence are proved in  Atkinson-Potra \cite{KEA1} for sufficiently smooth kernel of $K$:
\begin{equation*}
	\norm{\varphi - \varphi_{n}^C}_\infty = O\left(h^{2r +1}\right),  \quad \norm{\varphi - \varphi_{n}^S}_\infty = O\left(h^{4r +2}\right).
\end{equation*}
The following convergence rates are established in Grammont et al. \cite{LGRPK2}:
\begin{equation*}
	\norm{\varphi - \varphi_n^M }_\infty = O(h^{6r + 3}), \quad \norm{\varphi - \widetilde{\varphi}_n^M }_\infty = O(h^{8r + 4}).
\end{equation*}
In the case of interpolatory projection with collocation points which are not necessarily Gauss points, but $2r +1$ distinct special points in each subinterval as collocation points, and sufficiently smooth kernels, Grammont et al \cite{LGRPK2} proved that the above modified projection solutions (both) are of the order $h^{4r + 2}$. We observe that selecting Gauss points as collocation points improves the convergence orders through iterative refinement. 

In Rakshit et al \cite{GR-SKS-ASR}, we analyzed the orders of convergence for the approximate solutions of Fredholm integral equations with Green's kernel obtained by the collocation method, iterated collocation method and its modified variants. In the same paper, we considered $\mathcal{X}_n$ as the piecewise polynomial space of degree $\leq 2r$ ($r \geq 1$) with respect to the same uniform partition, and the projections as interpolatory operator at $2r+1$ midpoints of each subinterval $[t_{j-1}, t_j]$ for $j=1, 2, \ldots, n$. Then, we obtain the orders of convergence of collocation, iterated collocation, modified collocation and iterated modified collocation solutions as $h^{2r+1}, h^{2r+2}, h^{2r+2}$ and $h^{2r+3}$ respectively.

In this paper, we consider Urysohn integral equations with smooth and non-smooth kernels, and the projection $Q_n : \mathcal{X} \mapsto \mathcal{X}_n$ defined as interpolation at $2r+1$ equidistant points (including boundary points) in each subinterval $[t_{j-1}, t_j]$ for $j=1, 2, \ldots, n$. Next, the approximate solutions of \eqref{Eq: 02} using collocation methods and their variants are derived, and their orders of convergence are analyzed.

In many recent studies , such as \cite{ RPKTJN, RPKGRD1, NNPDGN,  RNNNSC},  improved convergence is typically achieved by using collocation points that are zeros of special functions. These choices are known to enhance the accuracy of numerical integration and are commonly used in classical iterated collocation methods.
In contrast, our method achieves higher-order convergence by using even-degree piecewise polynomial spaces with equidistant collocation points, which are not necessarily tied with zeros of special functions. This simpler setup still delivers strong performance, both theoretically and numerically, demonstrating the flexibility and effectiveness of our approach.

This article is organized as follows. Subsection $2.1$ presents the assumptions concerning the kernel of the Urysohn integral operator, preliminary results, and notations. Subsections $2.2$ and $2.3$ focus on interpolatory projection and few important lemmas using divided difference. Section $3$ discusses the methods of approximation including significant proofs and main theorems. Subsections $3.1$, $3.2$, and $3.3$ establish the orders of convergence for collocation and iterated collocation, modified collocation, and iterated modified collocation methods, respectively. Section $4$ provides numerical results, and Section $5$ draws a conclusion of this paper.
\section{Preliminaries}

For an integer $\alpha \geq 0$, let $C^\alpha[0,1]$ represent the space of all real-valued functions on  $[0,1]$ which are \( \alpha \)-times continuously differentiable, equipped with the norm
$$
\norm{x}_{\alpha,\infty} = \max_{0 \leq i \leq \alpha} \norm{x^{(i)}}_{\infty}  = \max_{0 \leq i \leq \alpha} \left\{ \sup_{0 \leq t \leq 1} |x^{(i)}(t)| \right\}.
$$
where $x^{(i)}$ is the $i$-th derivative of the function $x$.
\subsection{Green's function type kernel}

Consider the Urysohn integral operator \eqref{Eq: 01} as
\[
K(x)(s) = \int_0^1 \kappa(s, t, x(t)) \, dt, \quad s \in [0, 1], \quad x \in \mathcal{X},
\]
where the kernel $\kappa\in C(\Omega)$. 
Let $\alpha \ge 1$ be an integer. We assume that the kernel $\kappa$ of $K$ defined in equation \eqref{Eq: 01} satisfies the following properties:

\begin{enumerate}
	\item For $i=1,2,3,4$, the functions $\kappa, ~\frac{\partial^{i} \kappa}{\partial u^{i}} \in C(\Omega).$
	\item Divide $\Omega$ into two parts as $\Omega_1 = \{(s, t, u) : 0 \leq t \leq s \leq 1, u \in \mathbb{R}\}$
	and $\Omega_2 = \{(s, t, u) : 0 \leq s \leq t \leq 1, u \in \mathbb{R}\}$. There are two functions $\kappa_i \in C^\alpha(\Omega_i)$, $i = 1, 2$, such that
	$$
	\kappa(s, t, u) = 
	\begin{cases} 
		\kappa_1(s, t, u) & \text{if } (s, t, u) \in \Omega_1, \\
		\kappa_2(s, t, u) & \text{if } (s, t, u) \in \Omega_2.
	\end{cases}
	$$
	\item Let
	$$
	\ell(s, t, u) = \frac{\partial \kappa}{\partial u}(s, t, u), \quad m(s, t, u) = \frac{\partial^2 \kappa}{\partial u^2}(s, t, u), \quad \text{for all} \; (s,t,u) \in \Omega.
	$$
	Clearly, $\ell$, $m \in C(\Omega)$. The partial derivatives of $\ell(s,t,u)$ and $m(s,t,u)$ with respect to $s$ and $t$ exhibit jump discontinuities along the line $s = t$.
	
	\item There exist functions $\ell_i$, $m_i \in C^\beta(\Omega_i)$, $i = 1, 2$, such that
	$$
	\ell(s, t, u) = 
	\begin{cases} 
		\ell_1(s, t, u) & \text{if } (s, t, u) \in \Omega_1, \\
		\ell_2(s, t, u) & \text{if } (s, t, u) \in \Omega_2
	\end{cases}
	$$
	and
	$$
	m(s, t, u) = 
	\begin{cases} 
		m_1(s, t, u) & \text{if } (s, t, u) \in \Omega_1, \\
		m_2(s, t, u) & \text{if } (s, t, u) \in \Omega_2.
	\end{cases}
	$$
\end{enumerate}

Since $\frac{\partial^{i} \kappa}{\partial u^{i}} \in C(\Omega)$ for $i = 1, 2, 3, 4$, the operator $K$ being four times Fr\'echet differentiable at $x \in \mathcal{X}$ is expressed as
\begin{equation*} 
	K^{(i)}(x)( y_1,\ldots , y_i)(s) = \int_0^1 \frac{\partial^{i}\kappa}{\partial u^{i}} (s, t, x(t)) y_1(t) \cdots y_i(t) ~ dt, \quad s \in [0, 1],
\end{equation*}
where 
$$
\frac{\partial^i \kappa}{\partial u^i}(s, t, x(t)) = \left.\frac{\partial^i \kappa}{\partial u^i}(s, t, u) \right|_{u = x(t)}, \quad i = 1, 2, 3, 4,
$$
and $y_1, y_2, y_3, y_4 \in \mathcal{X}$. Note that $K'(x) : \mathcal{X} \mapsto \mathcal{X}$ is a compact linear operator, while $K^{(i)}(x) : \mathcal{X}^i \mapsto \mathcal{X}$ are multilinear operators, where $\mathcal{X}^i$ represents the Cartesian product of $i$ copies of $\mathcal{X}$ (see Rall \cite{RALL}). The norms of these operators are given by
$$
\norm{K^{(i)}(x)} = \sup \left\{ \norm{K^{(i)}(x)(y_1, \dots, y_i)}_\infty : \norm{y_j}_\infty \leq 1, \, j = 1, \dots, i \right\}
$$
for $i = 1, 2, 3, 4$. It follows that
$$
\norm{K^{(i)}(x)} \leq \sup_{s,t \in [0,1]} \left|\frac{\partial^i \kappa}{\partial u^i}(s,t,x(t))\right|, \quad i = 1, 2, 3, 4.
$$

\noindent
We use the following notations throughout the article:\\
Define $S_1 = \{(s, t) : 0 \le t \le s \le 1\}$, $S_2 = \{(s, t) : 0 \le s \le t \le 1\}$ and
\begin{align*}
	\ell_*(s, t) := \ell(s, t, \varphi(t)) =
	\begin{cases}
		\ell_{1,*}(s, t) = \ell_1(s, t, \varphi(t)), ~~~ & (s, t) \in S_1, \\
		\ell_{2,*}(s, t) = \ell_2(s, t, \varphi(t)), ~~~ & (s, t) \in S_2.
	\end{cases}
\end{align*}
Assume that $\ell_* \in C([0, 1] \times [0, 1])$, $\ell_{1,*}\in C^{2r +1}(S_1)$ and $\ell_{2,*} \in C^{2r+1}(S_2)$. 

For $j, k$ non-negative integers, we set
$$
D^{(j,k)}\ell_{i, *}(s,t)=\frac{\partial^{j+k} \ell_{i, *}}{\partial s^j \partial t ^k} (s,t), \quad \text{for} \; i=1,2.
$$
Define 
\begin{equation*}
	C_1 = \max_{0 \leq i \leq 4}\left(\sup_{\substack{s,t \in [0,1] \\ |u| \le \norm{\varphi}_\infty +\delta_0}} \left| \dfrac{\partial^{i} \kappa}{\partial u^{i}}(s,t,u)\right|\right), \quad	C_2 =  \max_{\substack{1 \le k \le n \\ j = 1, 2}}\bigg \{ \sup_{t_{k-1} \le t \le s \le t_k} \left| D^{(j,0)} \ell_{1, *}(s,t)\right|, \sup_{t_{k-1} \le  s \le t \le t_k} \left| D^{(j,0)} \ell_{2,*}(s,t)\right| \bigg \}.
\end{equation*}	

From the Fr\'echet derivatives of $K$, it is easy to see that
\begin{equation} \label{Eq: 04}
	\norm{(K'(\varphi)y)^{(i)}}_\infty \le C_1 \norm{y}_\infty, \; i=0,1,2,
\end{equation}
\begin{equation} \label{Eq: 05}
	\norm{(K''(\varphi)(y_1,y_2))^{(i)}}_\infty \le C_1  \norm{y_1}_\infty \norm{y_2}_\infty, \; i=0,1,2.
\end{equation}

We have
\begin{equation*}
	(K'(\varphi) y)(s) = \int_0^1 \frac{\partial\kappa}{\partial u}(s, t, \varphi(t)) y(t) ~dt = \int_0^1 \ell(s, t, \varphi(t)) y(t) ~dt, \quad s \in [0, 1].
\end{equation*}
For $\delta_0 > 0$, let $\mathcal{B}(\varphi,\delta_0)= \{x \in \mathcal{X} : \norm{x - \varphi}_\infty < \delta_0\}$ be a closed ball with centre $\varphi$ and radius $\delta_0$.
If $x, y \in \mathcal{B}(\varphi,\delta_0)$, then it can be easily shown that
\begin{equation} \label{Eq: 03}
	\norm{K'(x) - K'(y)} \leq \gamma \norm{x - y }_\infty,
\end{equation}
for some constant $\gamma$.

\subsection{Interpolatory projection}

Let the uniform partition of $[0,1]$ be 
\begin{equation*}
	\left\{ 0 =t_0 < t_1 < \cdots <t_n=1 \right\},
\end{equation*} 
where $t_j = \frac{j}{n} = j h, \quad j=1,2, \ldots, n.$ Denote $\Delta_{j} = [t_{j-1}, t_j]$. For $r \ge 0$, define
$
\mathcal{X}_n = \{x \in L^\infty[0, 1] : x|_{\Delta_{j}} \text{ is a polynomial of even degree } \leq 2r\}.
$
In each sub-interval $[t_{j-1}, t_j]$, consider $2r+1$ interpolation points as
\begin{equation*}
	\tau_j^i = t_{j-1} + \frac{ih}{2r}, \quad i = 0, 1, \ldots, 2r  \; \text{ for } \; r \geq 1, \quad \text{and} \quad \tau_j = \frac{t_{j-1} + t_j}{2} \; \text{ for }  \; r = 0.
\end{equation*}
Define $Q_n : \mathcal{X} \mapsto \mathcal{X}_n$ by
\begin{equation} \label{Eq: 06}
	Q_nx(\tau_j^i) = x(\tau_j^i), \quad i = 0, 1, 2, \ldots, 2r; ~ j = 1, 2, \ldots, n.
\end{equation}
Let $x \in C[0, 1]$, then the $(2r+1)^{\text{th}}$ divided difference of $x$ at $\tau_j^0, \tau_j^1, \dots, \tau_j^{2r}$ is defined by
\begin{equation*}
	\left[\tau_j^0, \tau_j^1, \dots, \tau_j^{2r}\right] x= \frac{\left[\tau_j^1, \tau_j^2, \dots, \tau_j^{2r}\right] x - \left[\tau_j^0, \tau_j^1, \dots, \tau_j^{2r-1}\right] x}{\tau_j^{2r}- \tau_j^0}.
\end{equation*}
Let 
\begin{equation*}
	\Psi_j(t) = \left(t - \tau_j^0 \right)\left(t - \tau_j^1 \right) \cdots \left(t - \tau_j^{2r} \right), \quad j =1,2, \ldots, n.
\end{equation*}
As $Q_{n,j}x$ (where $Q_{n, j}$ is $Q_n$ restricted on ${\Delta_j}$)  interpolates $x$ at the points $\tau_j^0, \tau_j^1, \ldots, \tau_j^{2r},$ therefore
\begin{equation*}
	x(t) - Q_{n,j}x(t) = \Psi_j(t) \left[\tau_j^0, \tau_j^1, \dots, \tau_j^{2r}, t  \right] x, \quad t \in \Delta_j.
\end{equation*} 
If $x \in C^{2r+1}(\Delta_j)$, then we have
\begin{equation} \label{Eq: 07}
	\norm{(I - Q_{n,j})x}_\infty \le C_3 \norm{x^{(2r+1)}}_\infty h^{2r+1},
\end{equation}
where $x^{(2r+1)}$ denotes the $(2r+1)^{\text{th}}$ derivative of $x$ and $C_3$ is a constant independent of $h$. See \cite{KEA0}.

\subsection{Even degree polynomial interpolation}
The scenario of interpolatory projection at equidistant collocation points (not necessarily the Gauss points) for a Green's type kernel of the Urysohn integral operator onto a space of piecewise polynomials of degree  $\le 2r$ over a uniform partition of $[0, 1]$ has not been explored in the research literature. This article addresses that gap by studying this case. We first prove two important lemmas based on the divided difference of $(K'(\varphi)x)$ at $\tau_j^0, \tau_j^1, \dots, \tau_j^{2r}$ and $s$ denoted by $\left[\tau_j^0, \tau_j^1, \dots, \tau_j^{2r}, s  \right](K'(\varphi)x)$.
\begin{lemma} \label{l01}
	Let $r \ge 0$ and $\left\{\tau_j^0, \tau_j^1, \dots, \tau_j^{2r} \right\}$ be the set of collocation points in $[t_{j-1}, t_j]$ for $j = 1, 2, \dots, n$. Then
	\begin{eqnarray*}
		\sup_{s \in [t_{j-1}, t_j]}	\left|\left[\tau_j^0, \tau_j^1, \dots, \tau_j^{2r}, s  \right](K'(\varphi)x)\right| \leq C_6 \norm{x}_\infty h^{-2r},
	\end{eqnarray*}
	for some constant $C_6$.
\end{lemma}

\begin{proof}    Recall that
	\begin{equation*}
		\tau_j^i = t_{j-1} + \frac{ih}{2r} = t_{j-1}+h\zeta_i, ~ \text{ for } i = 0, 1, 2, \dots, 2r; ~ j = 1, 2, \dots, n,
	\end{equation*}
	where $\displaystyle{\zeta_{i} = \frac{i}{2r} \in [0, 1]}$.
	Clearly $\tau_j^i \in [t_{j-1},t_j]$, we have
	$$
	(K'(\varphi)x)(s) = \int_0^1 \frac{\partial\kappa}{\partial u}(s, t, \varphi(t)) x(t) ~ dt = \int_0^1 \ell_*(s,t) x(t) ~ dt,  \quad s \in [0, 1],
	$$
	where
$\ell_*(s,t) ~= ~
\begin{cases}
	\ell_{1,*}(s,t),    ~~~     & 0\leq t\leq s\leq 1,\\
	\ell_{2,*}(s,t),    ~~~     & 0\leq s\leq t\leq 1.\\
\end{cases}$	Let $s \in [t_{j-1},t_j]$, then	 
	\begin{align} \label{Eq: 08}
		[\tau_j^0,\tau_j^1,s](K'(\varphi)x)&=\int_0^1 [\tau_j^0,\tau_j^1,s] \ell_*(.,t) x(t) ~dt \nonumber \\
		&= \underset{k \ne j}{\sum_{k=1}^{n}}\int_{t_{k-1}}^{t_k} [\tau_j^0,\tau_j^1,s] \ell_*(.,t) x(t) ~dt + \int_{t_{j-1}}^{t_j} [\tau_j^0,\tau_j^1,s] \ell_*(.,t) x(t) ~dt. 
	\end{align}
	Since $\kappa$ is sufficiently differentiable in the interval $[t_{k-1},t_k]$ for $k \ne j$,
	\begin{equation}\label{Eq: 09}
		\left | \int_{t_{k-1}}^{t_k} [\tau_j^0,\tau_j^1,s] \ell_*(.,t) x(t) dt \right | \le C_2 \norm{x}_\infty h.
	\end{equation}
	For $\tau_j^0, \tau_j^1, s \in [t_{j-1},t_j]$, we need to find bound for $[\tau_j^0,\tau_j^1,s] \ell_*(.,t)$.	Fix $s \in [t_{j-1},{t_j}].$
	
	\noindent
	\textbf{Case (I)}
	When $t_{j-1} \le s < \tau_j^0$. Note that 
	\begin{equation*}
		[\tau_j^0,s]\ell_*(.,t) =\dfrac{\ell_*(s,t)-\ell_*(\tau_j^0,t)}{s-\tau_j^0} ~=~
		\begin{cases}
			\dfrac{\ell_{1, *}(s,t)-\ell_{1,*}(\tau_j^0,t)}{s-\tau_j^0}  ~ &\text{if}\; t_{j-1} \le t \le s,\\
			\dfrac{\ell_{2,*}(s,t)-\ell_{1,*}(\tau_j^0,t)}{s-\tau_j^0}  ~ &\text{if}\; s < t < \tau_j^0,\\
			\dfrac{\ell_{2,*}(s,t)-\ell_{2,*}(\tau_j^0,t)}{s-\tau_j^0}  ~ &\text{if}\; \tau_j^0 \le t \le t_j.
		\end{cases}
	\end{equation*}	
	Thus, for a fixed $s \in [t_{j-1},\tau_j^0]$, the function $[\tau_j^0,s]\ell_*(.,t)$ is continuous on $[t_{j-1},t_j].$ Since $\ell_{1,*}(.,t)$ and $\ell_{2,*}(.,t)$ are continuous on $[s, \tau_j^0]$ and differentiable on $(s, \tau_j^0)$, by mean value theorem, we have
	\begin{equation*}
		\frac{\ell_{1, *}(s,t)-\ell_{1,*}(\tau_j^0,t)}{s-\tau_j^0} = D^{(1,0)}\ell_{1,*}(\eta_1,t),  ~~\text{and}~~ \frac{\ell_{2,*}(s,t)-\ell_{2,*}(\tau_j^0,t)}{s-\tau_j^0} = D^{(1,0)}\ell_{2,*}(\eta_2,t);
	\end{equation*}
	for some $\eta_1, \eta_2 \in (s, \tau_j^0).$
	Now, for $t_{j-1} \le t \le s$,
	\begin{equation*}
		\bigg | \frac{\ell_{1, *}(s,t)-\ell_{1,*}(\tau_j^0,t)}{s-\tau_j^0} \bigg|\le \bigg (\sup_{t_{j-1} \le t \le s \le t_j} \left| D^{(1,0)} \ell_{1, *}(s,t)\right| \bigg ) = C_2. 
	\end{equation*}
	Similarly, for $\tau_j^0 \le t \le t_j$,
	\begin{equation*}
		\bigg | \frac{\ell_{2,*}(s,t)-\ell_{2,*}(\tau_j^0,t)}{s-\tau_j^0} \bigg|\le \bigg (\sup_{t_{j-1} \le s \le t \le t_j} \left| D^{(1,0)} \ell_{2,*}(s,t)\right| \bigg ) = C_2. 
	\end{equation*}
	Hence, 
	\begin{equation}\label{Eq: 010}
		\left|[\tau_j^0,s]\ell_*(.,t)\right| \le C_2 , \quad \text{if} \; t_{j-1} \le t \le s  ~\text{and} ~ \tau_j^0 \le t \le t_j.
	\end{equation}
	Now, for $s < t < \tau_j^0$, first we find
	\begin{align*}
		\frac{\ell_{2,*}(s,t)-\ell_{1,*}(\tau_j^0,t)}{s-\tau_j^0} &= \frac{\ell_{2,*}(s,t)-\ell_{2,*}(t,t)+\ell_{1,*}(t,t)-\ell_{1,*}(\tau_j^0,t)}{s-\tau_j^0} \\
		&= \frac{D^{(1,0)}\ell_{2,*}(\eta_3,t)(s-t)}{s-\tau_j^0}+\frac{D^{(1,0)}\ell_{1,*}(\eta_4,t)(t-\tau_j^0)}{s-\tau_j^0},
	\end{align*}
	for some $\eta_3 \in (s,t)$ and $\eta_4 \in (t, \tau_j^0).$ Hence, 
	\begin{equation*}
		\bigg | \frac{\ell_{2,*}(s,t)-\ell_{1,*}(\tau_j^0,t)}{s-\tau_j^0} \bigg| \le  2 C_2.
	\end{equation*}
	Therefore,
	\begin{equation*}
		\left|[\tau_j^0,s]\ell_*(.,t)\right| \le 2C_2 , \quad \text{if} \; s < t < \tau_j^0.
	\end{equation*}
	Thus, by \eqref{Eq: 010} and above inequality, we obtain
	\begin{equation}\label{Eq: 011}
		\sup_{t \in [t_{j-1},t_j]}\left|[\tau_j^0,s]\ell_*(.,t)\right| \le 2C_2.
	\end{equation}
	
	\noindent
	\textbf{Case (II)}
	When $s = \tau_j^0$. In this case
	\begin{align*}
		[\tau_j^0,\tau_j^0] \ell_*(.,t) = \begin{cases}
			D^{(1,0)}\ell_{1,*}(\tau_j^0,t)  \quad &\text{if}\; t_{j-1} \le t < \tau_j^0 < 1,\\
			D^{(1,0)}\ell_{2,*}(\tau_j^0,t)  \quad &\text{if}\; t_{j-1} \le \tau_j^0 < t  < 1.\\	
		\end{cases} 
	\end{align*}
	Hence, 
	\begin{equation}\label{Eq: 012}
		\sup_{t \in [t_{j-1},t_j]}\left|[\tau_j^0,\tau_j^0]\ell_*(.,t)\right| \le C_2.
	\end{equation}
	
	\noindent
	\textbf{Case (III)}
	When $\tau_j^0<s \le t_j$. Proceeding the same way as in Case (I), we have
	\begin{equation*}
		\sup_{t \in [t_{j-1},t_j]}\left|[\tau_j^0,s]\ell_*(.,t)\right| \le 2C_2.
	\end{equation*}
	Hence, by \eqref{Eq: 011}, \eqref{Eq: 012} and above inequality, we obtain
	\begin{equation}\label{Eq: 013}
		\sup_{s,t \in [t_{j-1},t_j]}\left|[\tau_j^0,s]\ell_*(.,t)\right| \le 2C_2.
	\end{equation}
	Note that
	\begin{equation*}
		[\tau_j^0,\tau_j^1,s]\ell_*(.,t) =\dfrac{[\tau_j^0,s]\ell_*(.,t)-[\tau_j^1,s]\ell_*(.,t)}{\tau_j^0-\tau_j^1} = \dfrac{1}{h(\zeta_0-\zeta_1)} \bigg [ [\tau_j^0,s]\ell_*(.,t)-[\tau_j^1,s]\ell_*(.,t) \bigg ],
	\end{equation*}
	which gives
	\begin{align*}
		\bigg |[\tau_j^0,\tau_j^1,s]\ell_*(.,t) \bigg| &\le \dfrac{1}{h |\zeta_0-\zeta_1|} \bigg [ \left |[\tau_j^0,s]\ell_*(.,t)\right |+\left |[\tau_j^1,s]\ell_*(.,t) \right | \bigg ].
	\end{align*}
	It follows from \eqref{Eq: 013} that
	\begin{equation*}
		\sup_{s,t \in [t_{j-1},t_j]}\left|[\tau_j^0,\tau_j^1,s]\ell_*(.,t)\right| \le \dfrac{4C_2}{h |\zeta_0-\zeta_1|}.
	\end{equation*}
	Therefore, 
	\begin{equation*}
		\bigg |\int_{t_{j-1}}^{t_j} [\tau_j^0,\tau_j^1,s] \ell_*(.,t) x(t) ~dt \bigg | 
		\le \dfrac{4C_2}{h |\zeta_0-\zeta_1|}\norm{x}_\infty =  \dfrac{C_4}{h}\norm{x}_\infty,
	\end{equation*} 
	where $C_4=\dfrac{4C_2}{|\zeta_0-\zeta_1|}$ is some constant. Hence, using \eqref{Eq: 09} and above inequality in \eqref{Eq: 08}, we obtain
	\begin{equation*}
		\left|[\tau_j^0,\tau_j^1,s](K'(\varphi)x)\right| \le \bigg ( C_2 h + \dfrac{C_4}{h}\bigg)\norm{x}_\infty \le \dfrac{C_5}{h}\norm{x}_\infty,
	\end{equation*}
	where $C_5$ is some constant. Using mathematical induction, this gives the required estimate
	\begin{equation*}
		\left|[\tau_j^0,\tau_j^1,\dots, \tau_j^{2r}, s](K'(\varphi)x) \right| \le C_6 \norm{x}_\infty \dfrac{1}{h^{2r}},
	\end{equation*}
	for some constant $C_6.$ Since $s \in [t_{j-1},t_j]$ is arbitrary, the proof is complete.
\end{proof}
\begin{lemma} \label{l02}
	Let $r \ge 0$ and $\left\{\tau_j^0, \tau_j^1, \dots, \tau_j^{2r} \right\}$ be the set of collocation points in $[t_{j-1}, t_j]$ for $j = 1, 2, \dots, n$. Then
	\begin{eqnarray*}
		\sup_{s \in [t_{j-1}, t_j]}	\left|\left[\tau_j^0, \tau_j^1, \dots, \tau_j^{2r}, s, s \right] \left( K'(\varphi)x \right) \right| \leq C_8 \norm{x}_\infty h^{-2r},
	\end{eqnarray*}
	for some constant $C_8.$
\end{lemma}

\begin{proof}  
	Let $s \in [t_{j-1},t_j]$, then
	\begin{align}\label{Eq: 014}
		[\tau_j^0,s,s](K'(\varphi)x) &=\int_0^1 [\tau_j^0,s,s] \ell_*(.,t) x(t) ~dt \nonumber \\
		&= \underset{k \ne j}{\sum_{k=1}^{n}}\int_{t_{k-1}}^{t_k} [\tau_j^0,s,s] \ell_*(.,t) x(t) ~dt + \int_{t_{j-1}}^{t_j} [\tau_j^0,s,s] \ell_*(.,t) x(t) ~dt.
	\end{align}
	Since $\kappa$ is sufficiently differentiable in the interval $[t_{k-1},t_k]$ for $k \ne j$,
	\begin{equation}\label{Eq: 015}
		\left | \int_{t_{k-1}}^{t_k} [\tau_j^0,s,s] \ell_*(.,t) x(t) dt \right | \le C_2 \norm{x}_\infty h.
	\end{equation}  
	For $\tau_j^0, s \in [t_{j-1},t_j]$, we need to find bound for $[\tau_j^0,s,s] \ell_*(.,t)$. Fix $s \in [t_{j-1},{t_j}].$\\
	\noindent
	\textbf{Case (I)}
	When $t_{j-1} \le s < \tau_j^0$. Note that
	\begin{equation*}
		[\tau_j^0,s,s]\ell_*(.,t) =\dfrac{[s,s]\ell_*(.,t)-[\tau_j^0,s]\ell_*(.,t)}{s-\tau_j^0}.
	\end{equation*}
	Thus,
	\begin{equation*}
		[\tau_j^0,s,s]\ell_*(.,t) ~=~
		\begin{cases}
			\dfrac{\frac{\partial \ell_{1, *}}{\partial s}(s,t)- \left ( \frac{\ell_{1, *}(s,t)-\ell_{1,*}(\tau_j^0,t)}{s-\tau_j^0}\right) }{s-\tau_j^0}  \quad &\text{if}\; t_{j-1} \le t \le s,\\
			\dfrac{\frac{\partial \ell_{2,*}}{\partial s}(s,t)- \left ( \frac{\ell_{2,*}(s,t)-\ell_{1,*}(\tau_j^0,t)}{s-\tau_j^0}\right) }{s-\tau_j^0}  \quad &\text{if}\; s < t < \tau_j^0, \\
			\dfrac{\frac{\partial \ell_{2,*}}{\partial s}(s,t)- \left ( \frac{\ell_{2,*}(s,t)-\ell_{2,*}(\tau_j^0,t)}{s-\tau_j^0}\right) }{s-\tau_j^0}  \quad &\text{if}\; \tau_j^0 \le t \le t_j.
		\end{cases}
	\end{equation*}
	This function $[\tau_j^0,s,s]\ell_*(.,t)$ is possibly discontinuous at $t= s.$ Since $\ell_{1,*}(.,t)$ and $\ell_{2,*}(.,t)$ are continuous on $[s, \tau_j^0]$ and differentiable on $(s, \tau_j^0)$, by mean value theorem, we get	 
	\begin{equation*}
		\frac{\ell_{1, *}(s,t)-\ell_{1,*}(\tau_j^0,t)}{s-\tau_j^0} = D^{(1,0)}\ell_{1,*}(\nu_1,t); ~~\text{and}~~ \frac{\ell_{2,*}(s,t)-\ell_{2,*}(\tau_j^0,t)}{s-\tau_j^0} = D^{(1,0)}\ell_{2,*}(\nu_2,t),
	\end{equation*}
	for some $\nu_1, \nu_2 \in (s, \tau_j^0)$. Now, for $t_{j-1} \le t \le s,$	 
	\begin{align*}
	   \left|  \dfrac{\frac{\partial \ell_{1, *}}{\partial s}(s,t)- \left ( \frac{\ell_{1, *}(s,t)-\ell_{1,*}(\tau_j^0,t)}{s-\tau_j^0}\right )}{s-\tau_j^0}  \right| &= \Bigg | \dfrac{ D^{(1,0)}\ell_{1, *}(s,t)- D^{(1,0)}\ell_{1,*}(\nu_1,t)}{s-\tau_j^0} \Bigg | \nonumber \\
		&= \Bigg | \dfrac{ D^{(2,0)}\ell_{1,*}(\nu_3,t)(s-\nu_1)}{s-\tau_j^0} \Bigg |, ~\text{for} \; \nu_3 \in (s, \nu_1) \nonumber \\
		&\le \bigg (\sup_{t_{j-1} \le t \le s \le t_j} \left| D^{(2,0)} \ell_{1, *}(s,t)\right| \bigg ) = C_2. 
	\end{align*}
	Similarly, it can be shown that for $\tau_j^0 \le t \le t_j$,
	\begin{align*}
		\left|  \dfrac{\frac{\partial \ell_{2,*}}{\partial s}(s,t)- \left ( \frac{\ell_{2,*}(s,t)-\ell_{2,*}(\tau_j^0,t)}{s-\tau_j^0}\right )}{s-\tau_j^0}  \right| \le \bigg (\sup_{t_{j-1} \le s \le t \le t_j} \left| D^{(2,0)} \ell_{2,*}(s,t)\right| \bigg )= C_2.
	\end{align*}
	Hence, \begin{equation}\label{Eq: 016}
		\left|[\tau_j^0,s,s]\ell_*(.,t)\right| \le C_2, \quad \text{if} \; t_{j-1} \le t \le s ~ \text{and} ~ \tau_j^0 \le t \le t_j.
	\end{equation}
	Now, for $s < t < \tau_j^0$, we find
	\begin{align*}
		\frac{\ell_{2,*}(s,t)-\ell_{1,*}(\tau_j^0,t)}{s-\tau_j^0} &= \frac{\ell_{2,*}(s,t)-\ell_{2,*}(t,t)+\ell_{1,*}(t,t)-\ell_{1,*}(\tau_j^0,t)}{s-\tau_j^0} \\
		&= \frac{D^{(1,0)}\ell_{2,*}(\nu_5,t)(s-t)}{s-\tau_j^0}+\frac{D^{(1,0)}\ell_{1,*}(\nu_6,t)(t-\tau_j^0)}{s-\tau_j^0},
	\end{align*}
	for some $\nu_5 \in (s,t)$ and $\nu_6 \in (t, \tau_j^0)$. 	Then
	\begin{multline*}
		\Bigg | \frac{\partial \ell_{2,*}}{\partial s}(s,t)- \left ( \frac{\ell_{2,*}(s,t)-\ell_{1,*}(\tau_j^0,t)}{s-\tau_j^0}\right ) \Bigg |\\
		= \Bigg | D^{(1,0)}\ell_{2,*}(s,t)- \frac{D^{(1,0)}\ell_{2,*}(\nu_5,t)(s-t)}{s-\tau_j^0}-\frac{D^{(1,0)}\ell_{1,*}(\nu_6,t)(t-\tau_j^0)}{s-\tau_j^0}\Bigg |,
	\end{multline*}
	which gives
	\begin{equation*}
		\Bigg | \frac{\partial \ell_{2,*}}{\partial s}(s,t)- \left ( \frac{\ell_{2,*}(s,t)-\ell_{1,*}(\tau_j^0,t)}{s-\tau_j^0}\right ) \Bigg | \le 3C_2.
	\end{equation*}
	Hence,
	\begin{equation*}
		\bigintsss_{s}^{\tau_j^0}\Bigg | \dfrac{\frac{\partial \ell_{2,*}}{\partial s}(s,t)- \left ( \frac{\ell_{2,*}(s,t)-\ell_{1,*}(\tau_j^0,t)}{s-\tau_j^0}\right) }{s-\tau_j^0} \Bigg |~dt \le 3C_2.
	\end{equation*}
	Using \eqref{Eq: 016} and above inequality, we obtain
	\begin{equation*}
		\bigg |\int_{t_{j-1}}^{t_j} [\tau_j^0,s,s] \ell_*(.,t) x(t) ~dt \bigg | 
		\le \left ( C_2 +3 C_2 + C_2 \right )\norm{x}_\infty = 5 C_2 \norm{x}_\infty.
	\end{equation*} 
	Therefore, using \eqref{Eq: 015} and above estimate in \eqref{Eq: 014} we have
	\begin{equation}\label{Eq: 017}
		\left |[\tau_j^0,s,s](K'(\varphi)x) \right| \le \left ( C_2 h + 5 C_2 \right )\norm{x}_\infty \le 6 C_2 \norm{x}_\infty.
	\end{equation}
\\
	\noindent
	\textbf{Case (II)}
	Whenever $s = \tau_j^0$.
	\begin{align*}
		[\tau_j^0,\tau_j^0,\tau_j^0](K'(\varphi)x) &= \frac{1}{2} (K'(\varphi)x)^{''}(\tau_j^0)\\
		&= \frac{1}{2}\left ( \frac{\partial \ell_{1,*}}{\partial s}(\tau_j^0,\tau_j^0)-\frac{\partial \ell_{2,*}}{\partial s}(\tau_j^0,\tau_j^0)\right )x(\tau_j^0)\\&+\frac{1}{2}\left [ \int_{t_j-1}^{\tau_j^0}\frac{\partial^2 \ell_{1,*}}{\partial s^2}(\tau_j^0,t)x(t)~dt+\int_{\tau_j^0}^{t_j}\frac{\partial^2 \ell_{2,*}}{\partial s^2}(\tau_j^0,t)x(t)~dt\right ]. 
	\end{align*}
	Hence, 
	\begin{equation}\label{Eq: 018}
		\left |[\tau_j^0,\tau_j^0,\tau_j^0]K'(\varphi)x\right| \le 2C_2 \norm{x}_\infty. 
	\end{equation}
	\noindent
	\textbf{Case (III)}
	When $\tau_j^0<s \le t_j$. Similarly as Case (I), we have
	\begin{equation*}
		\left |[\tau_j^0,s,s](K'(\varphi)x) \right| \le \left ( C_2 h + 5 C_2 \right )\norm{x}_\infty \le 6 C_2 \norm{x}_\infty. 
	\end{equation*}
	From \eqref{Eq: 017}, \eqref{Eq: 018} and above inequality, it follows that
	\begin{equation}\label{Eq: 019}
		\sup_{s \in [t_{j-1},t_j]} \left|[\tau_j^0,s,s](K'(\varphi)x) \right| \le 6 C_2 \norm{x}_\infty,
	\end{equation}
	which proves the required estimate for the case $r = 0$. Now for $r=1$,
	\begin{align*}
		[\tau_j^0,\tau_j^1,s,s](K'(\varphi)x) 
		&= \dfrac{[\tau_j^0,s,s](K'(\varphi)x)-[\tau_j^1,s,s](K'(\varphi)x)}{\tau_j^0-\tau_j^1} \\
		&= \frac{1}{h (\zeta_0-\zeta_1) } \bigg [[\tau_j^0,s,s](K'(\varphi)x)-[\tau_j^1,s,s](K'(\varphi)x) \bigg],
	\end{align*}
	which gives
	\begin{equation*}
		\left |[\tau_j^0,\tau_j^1,s,s](K'(\varphi)x) \right | 	\le \frac{1}{h | \zeta_0-\zeta_1 |} \bigg [\left|[\tau_j^0,s,s](K'(\varphi)x) \right|+\left|[\tau_j^1,s,s](K'(\varphi)x) \right| \bigg].
	\end{equation*}
	Hence, it follows from \eqref{Eq: 019} that
	\begin{equation*}
		\sup_{s \in [t_{j-1},t_j]} \left|[\tau_j^0,\tau_j^1,s,s](K'(\varphi)x) \right| \le \frac{12 C_2}{h | \zeta_0-\zeta_1 |} \norm{x}_\infty = C_7 \norm{x}_\infty \frac{1}{h},
	\end{equation*}
	where $C_7 = \frac{12 C_2}{| \zeta_0-\zeta_1 |}$ is some constant. By mathematical induction, it follows that
	\begin{equation*}
		\sup_{s \in [t_{j-1},t_j]} \left|[\tau_j^0,\tau_j^1,\dots, \tau_j^{2r},s,s](K'(\varphi)x) \right| \le C_8 \norm{x}_\infty \frac{1}{h^{2r}} ,
	\end{equation*}
	where $C_8$ is some constant.
\end{proof}
\section{Iteration methods of approximation}

We see that $K$ is a Urysohn integral operator with Green's function type kernel and $Q_n$ is the interpolatory operator at $2r+1$ interpolating points defined by \eqref{Eq: 06}. In the collocation method, the main equation \eqref{Eq: 02} is approximated by
\begin{equation} \label{Eq: 020}
	\varphi_{n}^C - Q_nK(\varphi_{n}^C) = Q_nf.
\end{equation}
The iterated collocation solution is defined by
\begin{equation} 	\label{Eq: 021}
	\varphi_n^S = K (\varphi_n^S) + f. 
\end{equation}
In Grammont-Kulkarni \cite{LGRPK1}, the following modified collocation method is proposed:
\begin{equation} \label{Eq: 022}
	\varphi_n^M - K_n^M (\varphi_n^M)= f, 
\end{equation}
where
\begin{equation*}
K_n^M (x) = Q_n  K (x) + K (Q_n x) - Q_n K(Q_n x).
\end{equation*}
The iterated modified collocation solution is defined as 
\begin{equation} \label{Eq: 023}
	\widetilde{\varphi}_n^M = K (\varphi_n^M) + f.
\end{equation}

In this section, we present some preliminary results using conventional methodologies and the previously mentioned lemmas. These results significantly contribute to our main result.
\begin{proposition} \label{p01}
	If $x \in C^{2r+2}[0,1]$, then
	\begin{eqnarray*}
		\norm{K'(\varphi)(I-Q_n)x }_\infty \le 2C_{9} \norm{x}_{2r+2, \infty} h^{2r + 2},
	\end{eqnarray*}
	where $C_{9}$ is a constant independent of $h$.
\end{proposition}

\begin{proof} 	   Let $s \in [0,1]$. Then we have
	\begin{align*}
		K'(\varphi)(I-Q_n)x(s) &=\int_0^1 \ell_*(s,t) (I-Q_n) x(t) ~dt \notag \\
		&= \sum_{j=1}^{n} \int_{t_{j-1}}^{t_j} \ell_*(s,t) (I-Q_{n,j})x(t) ~dt  \notag \\
		&= \sum_{j=1}^{n} \int_{t_{j-1}}^{t_j} \ell_*(s,t) \Psi_j(t)  \left[\tau_j^0,\tau_j^1,\dots, \tau_j^{2r},t\right]x  ~dt,
	\end{align*}
	where
	\begin{equation*}
		\Psi_j(t) = \prod_{i=0}^{2r} (t- \tau_j^i), \quad \text{for} ~ j=1,2,\dots,n.
	\end{equation*} 
	\\
	Let $s \in [t_{i-1},t_i] \subset [0,1]$, for some $1 \le i \le n$. Then we write
	\begin{equation} \label{Eq: 024}
		K'(\varphi)(I-Q_n)x(s) = \sum_{\substack{j=1 \\  j\ne i} }^{n} 	\int_{t_{j-1}}^{t_j} \ell_*(s,t) \Psi_j(t) \left[\tau_j^0,\tau_j^1,\dots, \tau_j^{2r},t\right]x ~dt 
		  + \int_{t_{i-1}}^{t_i} \ell_*(s,t) \Psi_i(t) \left[\tau_i^0,\tau_i^1,\dots, \tau_i^{2r},t\right]x ~dt.		
	\end{equation}
	For fix $s$, let
	\begin{align*}
		f_j^s(t) = \ell_*(s,t) \left[\tau_j^0,\tau_j^1,\dots, \tau_j^{2r},t\right]x, \quad t \in [t_{j-1}, t_j], ~ \text{for} ~ j=1,2,\dots,n.
	\end{align*}
	That is
	\begin{equation*}
		f_j^s(t) ~= ~
		\begin{cases}
			\ell_{1,*}(s,t) \left[\tau_j^0,\tau_j^1,\dots, \tau_j^{2r},t\right]x,  ~~~          &  j \le i,\\
			\ell_{2,*}(s,t) \left[\tau_j^0,\tau_j^1,\dots, \tau_j^{2r},t\right]x, ~~~          & j \ge i.\\
		\end{cases} 
	\end{equation*} 
	Note that the function $f_j^s$ is continuous on $[t_{j-1}, t_j]$ and differentiable on $(t_{j-1}, t_j)$ for $j \ne i$. Therefore by Taylor series expansion of $f_j^s$ about the point $s_j = \frac{t_{j-1}+t_j}{2}$, there exist a point $\zeta_{j1} \in (t_{j-1}, t_j)$ such that
	\begin{equation*}
		(f_j^s)(t) = f_j^s \left( s_j \right) + \left(t - s_j \right) (f_j^s)'(\zeta_{j1}),
	\end{equation*} 
	where
	\begin{equation*}
		(f_j^s)'(\zeta_{j1}) ~= ~
		\begin{cases}
			\ell_{1,*}(s,\zeta_{j1}) \left[\tau_j^0,\tau_j^1,\dots, \tau_j^{2r},\zeta_{j1}, \zeta_{j1} \right]x 
			+ D^{(0,1)}\ell_{1,*}(s,\zeta_{j1}) \left[\tau_j^0,\tau_j^1,\dots, \tau_j^{2r},\zeta_{j1} \right]x,~~~          &  j \le i, \vspace*{1.5mm}\\
			\ell_{2,*}(s,\zeta_{j1}) \left[\tau_j^0,\tau_j^1,\dots, \tau_j^{2r},\zeta_{j1}, \zeta_{j1} \right]x 
			+  D^{(0,1)} \ell_{2,*}(s,\zeta_{j1}) \left[\tau_j^0,\tau_j^1,\dots, \tau_j^{2r},\zeta_{j1} \right]x,~~~          & j \ge i.\\
		\end{cases} 
	\end{equation*} 
	Let
	\begin{equation*}
		C_{9} = \max_{k =0,1,2}\bigg \{ \sup_{0 \le t \le s \le 1} \left| D^{(0,k)}\ell_{1,*}(s,t)\right|, \sup_{0 \le  s \le t \le 1} \left| D^{(0,k)} \ell_{2,*}(s,t)\right| \bigg \}.
	\end{equation*}  
	Then 
	\begin{equation*}
		\left|(f_j^s)'(\zeta_{j1}) \right| \le 2 C_{9} \norm{x}_{2r +2, \infty}.
	\end{equation*}
	Now, we write
	\begin{equation*}
		\int_{t_{j-1}}^{t_j} \ell_*(s,t) \Psi_j(t) \left[\tau_j^0,\tau_j^1,\dots, \tau_j^{2r},t\right]x ~dt = \int_{t_{j-1}}^{t_j}  \left(t - s_j \right) (f_j^s)'(\zeta_{j1}) \Psi_j(t) ~dt.
	\end{equation*}
	It follows that 
	\begin{equation*}
		\left|\int_{t_{j-1}}^{t_j} \ell_*(s,t) \Psi_j(t) \left[\tau_j^0,\tau_j^1,\dots, \tau_j^{2r},t\right]x ~dt \right| \le  2 C_{10} \norm{x}_{2r +2, \infty} h^{2r + 3}.
	\end{equation*}
	Since $\int_{t_{i-1}}^{t_i} \Psi_i(t) ~dt =0$,
	\begin{align*}
		\int_{t_{i-1}}^{t_i} \ell_*(s,t) \Psi_i(t) \left[\tau_i^0,\tau_i^1,\dots, \tau_i^{2r},t\right]x ~dt &= \int_{t_{i-1}}^{t_i}   (f_i^s)(t) \Psi_i(t) ~dt, ~ s \in [t_{i-1},t_i] \\
		&= \int_{t_{i-1}}^{t_i}   [(f_i^s)(t)-(f_i^s)(s)] \Psi_i(t) ~dt. 
	\end{align*}
	Since the function $f_i^s$ is continuous on $[t_{i-1}, t_i]$ and differentiable on $(t_{i-1}, t_i)$, by mean value theorem, we have 
	\begin{equation*}
		(f_i^s)(t)-(f_i^s)(s) = (t-s)(f_i^s)'(\zeta_{i2}),
	\end{equation*}
	for some $\zeta_{i2} \in (t_{i-1}, t_i)$. Then
	\begin{equation*}
		\left|\int_{t_{i-1}}^{t_i} \ell_*(s,t) \Psi_i(t) \left[\tau_i^0,\tau_i^1,\dots, \tau_i^{2r},t\right]x ~dt \right| \le  2 C_{10} \norm{x}_{2r +2, \infty} h^{2r + 3}.
	\end{equation*}
	By \eqref{Eq: 024},
	\begin{align*}
		\left|K'(\varphi)(I-Q_n)x(s)\right| &\le \sum_{\substack{j=1 \\  j\ne i} }^{n} 	\left|\int_{t_{j-1}}^{t_j} \ell_*(s,t) \left( [\tau_j^0,\tau_j^1,\dots, \tau_j^{2r},t]x \right) \Psi_j(t) ~dt \right|\\
		& \quad + \left|\int_{t_{i-1}}^{t_i} \ell_*(s,t) \left( [\tau_i^0,\tau_i^1,\dots, \tau_i^{2r},t]x \right) \Psi_i(t) ~dt 	\right| .
	\end{align*}
	Hence
	\begin{equation*}
		\norm{K'(\varphi)(I-Q_n)x(s)}_\infty \le 2 C_{9} \norm{x}_{2r +2, \infty} h^{2r + 2},
	\end{equation*}
	which proves the result.
\end{proof}
\begin{proposition} \label{p02}
	If $x \in C^{2r+1}[0,1]$, then
	\begin{eqnarray*}
		\norm{K'(\varphi)(I-Q_n)K'(\varphi)(I-Q_n)x }_\infty \le C_{11} \norm{x}_{2r+1, \infty} h^{2r + 3},
	\end{eqnarray*}
	where $C_{11}$ is a constant independent of $h$.
\end{proposition}

\begin{proof}  
	For any $s \in [0,1]$, 
		\begin{align} \label{Eq: 025}
			K'(\varphi)(I-Q_n)K'(\varphi)(I-Q_n)x(s) &=\int_0^1 \ell_*(s,t) (I-Q_n)K'(\varphi)(I-Q_n)x(t) ~dt \notag \\
			&= \sum_{j=1}^{n} \int_{t_{j-1}}^{t_j}\ell_*(s,t)(I-Q_{n,j})K'(\varphi)(I-Q_n)x(t)~dt\notag \\
			&= \sum_{j=1}^{n} \int_{t_{j-1}}^{t_j} g_j^s(t) \Psi_j(t) ~dt,
		\end{align}
	where
	\begin{align*}
		g_j^s(t) = \ell_*(s,t) \left( \left[\tau_j^0,\tau_j^1,\dots, \tau_j^{2r},t\right]K'(\varphi)(I-Q_n)x \right), ~ t \in [t_{j-1}, t_j], ~ \text{for} ~ j=1,2,\dots,n.
	\end{align*}
	Let $s \in [t_{i-1},t_i] \subset [0,1]$, for some $1 \le i \le n$.
	Note that
	\begin{equation*}
		g_j^s(t) ~= ~
		\begin{cases}
			\ell_{1,*}(s,t) \left( \left[\tau_j^0,\tau_j^1,\dots, \tau_j^{2r},t\right]K'(\varphi)(I-Q_n)x \right), ~~~          &  j \le i,\\
			\ell_{2, *}(s,t) \left( \left[\tau_j^0,\tau_j^1,\dots, \tau_j^{2r},t\right]K'(\varphi)(I-Q_n)x \right),~~~          & j \ge i.\\
		\end{cases} 
	\end{equation*} 
	Denote
	\begin{equation*}
		\left ( \delta_j^{2r} K'(\varphi)(I-Q_n)x \right ) [t] =\left( \left[\tau_j^0,\tau_j^1,\dots, \tau_j^{2r},t\right]K'(\varphi)(I-Q_n)x \right),
	\end{equation*}   
	and
	\begin{equation*}
		\left ( \delta_j^{2r}  K'(\varphi)(I-Q_n)x \right ) [t, t] =\left( \left[\tau_j^0,\tau_j^1,\dots, \tau_j^{2r},t, t \right]K'(\varphi)(I-Q_n)x \right).
	\end{equation*}  
	Note that the function $g_j^s$ is continuous on $[t_{j-1}, t_j]$ and differentiable on $(t_{j-1}, t_j)$ for $j \ne i$. Then it follows that
	\begin{equation*}
		(g_j^s)'(t) ~= ~
		\begin{cases}
			\ell_{1,*}(s,t) \left ( \delta_j^{2r}  K'(\varphi)(I-Q_n)x \right ) [t, t] 
			+ \left ( \delta_j^{2r}  K'(\varphi)(I-Q_n)x \right ) [t]~ D^{(0,1)}\ell_{1,*}(s,t), ~~~          &  j \le i, \vspace*{1.5mm}\\
			\ell_{2, *}(s,t) \left ( \delta_j^{2r}  K'(\varphi)(I-Q_n)x \right ) [t, t] 
			+ \left ( \delta_j^{2r}  K'(\varphi)(I-Q_n)x \right ) [t]~ D^{(0,1)} \ell_{2, *}(s,t),~~~          & j \ge i.\\
		\end{cases} 
	\end{equation*} 
	Therefore, by Lemma \ref{l01}, Lemma \ref{l02} and estimate \eqref{Eq: 07}, we have
	\begin{equation*}
		\left| (g_j^s)'(t) \right| \le C_{9} C_{10} \norm{(I-Q_n)x}_\infty h^{-2r}  \le C_3 C_{9} C_{10} \norm{x^{(2r + 1)}}_\infty h.
	\end{equation*}
	From \eqref{Eq: 025}, we write
	\begin{align*}
		K'(\varphi)(I-Q_n)K'(\varphi)(I-Q_n)x(s) &= \sum_{\substack{j=1 \\  j\ne i} }^{n} 	\int_{t_{j-1}}^{t_j}  g_j^s(t) \Psi_j(t) ~dt + \int_{t_{i-1}}^{t_i}  g_i^s(t) \Psi_i(t) ~dt
	\end{align*}
	By employing similar argument as presented in the proof of Proposition \ref{p01}, we acquire 
	\begin{equation*}
		\left|K'(\varphi)(I-Q_n)K'(\varphi)(I-Q_n)x(s)\right| \le 2 C_3 C_{9}  C_{10} \norm{x^{(2r + 1)}}_\infty h^{2r + 3}.
	\end{equation*}
	This proves the proposition.
\end{proof}
\begin{proposition}\label{p03}
	If $x \in C[0,1]$, then
	\begin{eqnarray*}
		\norm{K'(\varphi)(I-Q_n)K'(\varphi)x }_\infty \le C_{12} \norm{x}_{\infty} h^{ 2},
	\end{eqnarray*}
	where $C_{12}$ is a constant independent of $h$.
	Further, 
	\begin{eqnarray*}
		\norm{K'(\varphi)(I-Q_n)K'(\varphi) } =  O(h^{2}).
	\end{eqnarray*}	
\end{proposition}

\begin{proof}   Let $s \in [0,1]$. Then 
	\begin{align*}
		K'(\varphi)(I-Q_n)K'(\varphi)x(s) &=\int_0^1 \ell_*(s,t) (I-Q_n)K'(\varphi)x(t) ~dt  \\
		&= \sum_{j=1}^{n} \int_{t_{j-1}}^{t_j} \ell_*(s,t) (I-Q_{n,j})K'(\varphi)x(t) ~dt.
	\end{align*}
	Since $x \in C[0,1]$, then $K'(\varphi)x \in C^2[0,1]$, and therefore
	\begin{align*}
		| K'(\varphi)(I-Q_n)K'(\varphi)x(s) | &\le \norm{\ell_*}_\infty \sum_{j=1}^{n} \int_{t_{j-1}}^{t_j} \norm{(I-Q_{n,j})K'(\varphi)x}_\infty ~dt  \\
		&\le   \norm{\ell_*}_\infty \sum_{j=1}^{n} \int_{t_{j-1}}^{t_j} C_3 \norm{(K'(\varphi)x)^{(2)}}_\infty h^2 ~dt. 
	\end{align*}
	From the estimate \eqref{Eq: 04}, we have 
	\begin{align*}
		| K'(\varphi)(I-Q_n)K'(\varphi)x(s) | 
		&\le   \norm{\ell_*}_\infty \sum_{j=1}^{n} \int_{t_{j-1}}^{t_j} C_1 C_3  \norm{x}_\infty h^2 ~dt = C_1 C_3 \norm{\ell_*}_\infty \norm{x}_\infty h^2.
	\end{align*}
	This proves the required result, where $C_{12}= C_1 C_3  \norm{\ell_*}_\infty < \infty$. Thus,
	\begin{equation*}
		\norm{K'(\varphi)(I-Q_n)K'(\varphi)} = \sup_{\norm{x}_\infty \le 1}  \left\{\sup_{s \in [0,1] } | K'(\varphi)(I-Q_n)K'(\varphi)x(s) |\right\}= O\left(h^2\right), 
	\end{equation*} 
	which completes the proof.
\end{proof}
\begin{remark} \label{r01}
	If $x \in C[0,1]$, then it can be shown that
	$$\norm{K'(\varphi)(I-Q_n)K''(\varphi) } =  O\left(h^{2}\right),$$
	and 
	$$\norm{K'(\varphi)(I-Q_n)K^{(3)}(\varphi) } =  O\left(h^{2}\right).$$
\end{remark}

Further in this section, we will now establish convergence orders for collocation and iteration collocation solutions (i.e. $\varphi_n^C$ and $\varphi_n^S$). Then, in the next subsection, we will determine convergence orders for both the modified collocation solution $\varphi_n^M$ and the iterated modified collocation solution $\widetilde{\varphi}_n^M$.
\subsection{Collocation and iterated collocation method}
We find the orders of convergence for collocation solution of \eqref{Eq: 020} and iterated collocation solution defined as \eqref{Eq: 021} by using the following error estimates from Atkinson-Potra \cite{KEA1}, 
$$
\norm{\varphi - \varphi_n^C }_\infty \leq C_{13} \norm{(I - Q_n)\varphi}_\infty
$$
and
$$
\norm{\varphi - \varphi_n^S }_\infty \leq  C_{13} \norm{K'(\varphi)(I - Q_n)\varphi}_\infty,
$$
where $C_{13}$ is a constant independent of $n$. Therefore, by inequality \eqref{Eq: 07} and Proposition \ref{p01}, we have
\begin{equation*}
	\norm{\varphi - \varphi_n^C }_\infty = O\left(h^{2r + 1}\right), \quad \norm{\varphi - \varphi_n^S }_\infty = O\left(h^{2r + 2}\right).
\end{equation*}
In this case, the one step iteration improves the order of convergence. 
\begin{note} \label{n01}
	Let the Urysohn integral equation \eqref{Eq: 02} is uniquely solvable for all $f \in \mathcal{X}$ and $K$ be the Urysohn integral operator defined by \eqref{Eq: 01} with a smooth kernel such that $\ell_{*} \in C^{2r + 2}([0,1] \times [0,1])$ and $Q_n$ be the interpolatory operator onto the approximating space $\mathcal{X}_n$ defined by \eqref{Eq: 06}. Then, by Rane \cite{RANE}, it can be shown that
	$$
	\norm{\varphi - \varphi_n^C }_\infty = O\left(h^{2r + 1}\right), \quad \norm{\varphi - \varphi_n^S }_\infty = O\left(h^{2r + 2}\right).
	$$
\end{note}
\subsection{Modified collocation method}
We begin by establishing the following key result, which plays a crucial role in proving Theorem \ref{t01}, the main theorem of this section. Henceforth, we assume that $n_0$ is a positive integer such that for all $n \geq n_0$, it follows that $Q_n\varphi \in \mathcal{B}(\varphi, \delta_0)$.
\begin{lemma} \label{l03}
	Let $r \ge 1$ and $f \in C^{2r+1}[0,1]$. Then
	\begin{eqnarray*}
		\norm{(I-Q_n) \big [K(Q_n \varphi)-K(\varphi)-K'(\varphi)(Q_n \varphi -\varphi)\big ]}_\infty =  O\left(h^{4r+3}\right).
	\end{eqnarray*}
\end{lemma}

\begin{proof}   The generalized Taylor's theorem gives
	\begin{align*} 
		&K(Q_n \varphi)-K(\varphi)-K'(\varphi)(Q_n \varphi -\varphi) = \frac{1}{2} K''(\varphi)(Q_n \varphi -\varphi)^2  +R(Q_n \varphi -\varphi),
	\end{align*}
	with 
	\begin{equation}  \label {Eq: 026}
		R(Q_n \varphi -\varphi)(s)=\int_0^1 \left[K^{(3)}(\varphi + \theta (Q_n \varphi -\varphi))(Q_n \varphi -\varphi)^3\right](s) \dfrac{(1-\theta)^2}{2} ~ d\theta, \;\; s \in [0,1].
	\end{equation}
	Recall that $\dfrac{\partial^3 \kappa}{\partial u^3}$ is continuous on $\Omega$. Then the operator $K$ is three times Fr\'echet differentiable and the third Fr\'echet derivative of $K$ is given by,
	\begin{equation*} 
		\left[K^{(3)}(\varphi + \theta (Q_n \varphi -\varphi))(Q_n \varphi -\varphi)^3\right](s) 
		= \int_0^1 \dfrac{\partial^3 \kappa}{\partial u^3}(s,t,\varphi(t)+ \theta (Q_n \varphi -\varphi)(t)) (Q_n \varphi -\varphi)^3(t)~dt.
	\end{equation*}
	Then
	\begin{equation*} 
		\left|\left[K^{(3)}(\varphi + \theta (Q_n \varphi -\varphi))(Q_n \varphi -\varphi)^3\right](s)\right| \le \left( \sup_{\substack{s,t \in [0,1] \\ |u| \le \norm{\varphi}_\infty +\delta_0}}\left| \dfrac{\partial^3 \kappa}{\partial u^3}(s,t,u)\right| \right) \norm{Q_n \varphi -\varphi}_{\infty}^3.
	\end{equation*}
	Thus,
	\begin{equation*} 
		\norm{\left[K^{(3)}(\varphi + \theta (Q_n \varphi -\varphi))(Q_n \varphi -\varphi)^3\right]}_\infty \le C_1 \norm{Q_n \varphi -\varphi}_\infty^3.
	\end{equation*}
	Therefore, equation \eqref{Eq: 026} gives
	\begin{equation} \label {Eq: 027}
		\norm{ R(Q_n \varphi -\varphi)}_\infty \le \dfrac{C_1}{6} \norm{Q_n \varphi -\varphi}_\infty^3.
	\end{equation}
	Note that
	\begin{equation} \label {Eq: 028}
		(I-Q_n) [K(Q_n \varphi)-K(\varphi)-K'(\varphi)(Q_n \varphi -\varphi)] =  \frac{1}{2}(I-Q_n) K''(\varphi)(Q_n \varphi -\varphi)^2  + (I-Q_n)R(Q_n \varphi -\varphi).   
	\end{equation}
	Since $f \in C^{2r+1}[0,1]$, $\varphi \in C^{2r+1}[0,1]$. Thus $\norm{Q_n \varphi -\varphi}_\infty = O\left(h^{2r+1}\right)$. As $Q_n \varphi -\varphi \in C[0,1]$, therefore $K''(\varphi)(Q_n \varphi -\varphi)^2 \in C^2[0,1]$. Hence,
	\begin{align*} 
		\norm{(I-Q_n) K''(\varphi)(Q_n \varphi -\varphi)^2}_\infty \le C_3 \norm{( K''(\varphi) \left(Q_n \varphi -\varphi)^2\right)'}_\infty h.
	\end{align*}
	From \eqref{Eq: 05}, we have
	\begin{align*}
		\norm{( K''(\varphi)\left(Q_n \varphi -\varphi)^2\right)'}_\infty \le C_1 \norm{Q_n \varphi -\varphi}_\infty^2.
	\end{align*}
	Therefore,
	\begin{equation*} 
		\norm{(I-Q_n) K''(\varphi)(Q_n \varphi -\varphi)^2}_\infty = O\left(h^{4r +3}\right).
	\end{equation*}
	On the other hand, since $\left\{Q_n \right\}$ is uniformly bounded and using \eqref{Eq: 027}, we obtain
	\begin{equation*} 
		\norm{(I-Q_n)R(Q_n \varphi -\varphi)}_\infty = O\left(h^{6r+3}\right).
	\end{equation*}
	Hence, by employing the above two inequalities in \eqref{Eq: 028}, the required estimate as
	\begin{equation*} 
		\norm{(I-Q_n) \big [K(Q_n \varphi)-K(\varphi)-K'(\varphi)(Q_n \varphi -\varphi)\big]}_\infty =  O\left(h^{4r+3}\right).
	\end{equation*}
\end{proof}


\begin{theorem} \label{t01}
	Let $K$ be a Urysohn integral operator defined as \eqref{Eq: 01} with Green's function type kernel and $f \in C^{2r +2}[0,1]$ be a given function. Consider the equation \eqref{Eq: 02} with $\varphi$ as the unique solution in $\mathcal{B}(\varphi, \delta_0)$. Let $Q_n$ be the interpolatory operator onto the approximating space $\mathcal{X}_n$ defined by \eqref{Eq: 06}. Let $\varphi_n^M \in \mathcal{B}(\varphi, \delta_0)$ be the modified collocation solution of \eqref{Eq: 022}. Then
	\begin{eqnarray*}
		\norm{\varphi - \varphi_n^{M}}_\infty =  	\begin{cases}
			O\left( h^{3} \right), ~~ & r = 0, \\
			O\left( h^{2r+2} \right), ~~ & r \geq 1.
		\end{cases}
	\end{eqnarray*}
\end{theorem}

\begin{proof}Case I: When $r = 0$. See \cite{RPKTJN}.\\
	Case II: When $r \geq 1$. \\
	 We write the following estimate from Grammont et al \cite{LGRPK2}: 
	\begin{align} \label {Eq: 029}
		\norm{\varphi - \varphi_n^{M}}_\infty \le 4 \norm{(I-K'(\varphi))^{-1}} \norm{[K(\varphi) - K_n^{M}(\varphi)]}_\infty.
	\end{align} 
	Further consider, 
	\begin{align} \label {Eq: 030}
		K(\varphi)-K_n^{M}(\varphi) &= (I-Q_n)(K(\varphi)-K(Q_n\varphi)) \notag \\
		&=-(I-Q_n) [K(Q_n \varphi)-K(\varphi)-K'(\varphi)(Q_n \varphi -\varphi)]+(I-Q_n)K'(\varphi)(I-Q_n)\varphi.
	\end{align}
	Now, 
	\begin{align*} 
		\norm{(I-Q_n)K'(\varphi)(I-Q_n)\varphi}_\infty \le \norm{I-Q_n} \norm{K'(\varphi)(I-Q_n)\varphi}_\infty.
	\end{align*}
	 Uniformly boundedness $\left\{Q_n \right\}$ of and Proposition \ref{p01} implies
	\begin{align*} 
		\norm{(I-Q_n)K'(\varphi)(I-Q_n)\varphi}_\infty = O\left(h^{2r+2}\right).
	\end{align*}
	Applying above estimate and Lemma \ref{l03} in equation \eqref{Eq: 030}, we obtain
	\begin{align*} 
		\norm{	K(\varphi)-K_n^{M}(\varphi)}_\infty = O\left(h^{2r+2}\right).
	\end{align*}
	Using the above estimate in \eqref{Eq: 029}, we obtained the desired result as
	\begin{eqnarray*}
		\norm{\varphi - \varphi_n^{M}}_\infty =  O\left(h^{2r+2}\right).
	\end{eqnarray*}
This proves the theorem.	
\end{proof}
\begin{note} \label{n02}
	Let $K$ be the Urysohn integral operator defined by \eqref{Eq: 01} with a smooth kernel such that $\ell_{*} \in C^{4r +2}([0,1] \times [0,1])$ and $Q_n$ be the interpolatory operator onto the approximating space $\mathcal{X}_n$ defined by \eqref{Eq: 06}. If $f \in C^{4r}[0,1]$, then by Rane \cite{RANE}
\begin{equation*}
	\norm{(I - Q_n) K'(\varphi) (I -Q_n)\varphi}_\infty = O\left(h^{4r+3}\right),
\end{equation*}
	and from Grammont et al \cite[Lemma 2.1]{LGRPK2}, we get
	\begin{equation*}
		\norm{(I-Q_n) \big [K(Q_n \varphi)-K(\varphi)-K'(\varphi)(Q_n \varphi -\varphi)\big]}_\infty =  O\left(h^{6r+3}\right).
	\end{equation*}
	Thus, applying these estimates to \eqref{Eq: 029} first and subsequently to \eqref{Eq: 030}, we obtain
	$	\displaystyle{\norm{\varphi - \varphi_n^{M} }_\infty =  O\left(h^{4r+3}\right)}	$.
\end{note}

\subsection{Iterated modified collocation method}
In this subsection, we establish our primary result. We demonstrate that the convergence rate in the iterated modified collocation method is higher than the modified collocation methods. We proceed to prove a series of intermediate results essential for the proof of our main theorem. 

\begin{lemma} \label{l04}
	Let $r \ge 1$ and $f \in C^{2r+1}[0, 1]$, then
	\begin{eqnarray*}
		\norm{K'(\varphi)(I-Q_n) \big [K(Q_n \varphi)-K(\varphi)-K'(\varphi)(Q_n \varphi -\varphi)\big ]}_\infty =  O\left(h^{4r+4}\right).
	\end{eqnarray*}
\end{lemma}

\begin{proof}   Let $n_0$ be a positive integer such that $n \ge n_0$ implies that
	$Q_n \varphi \in \mathcal{B}(\varphi,\delta_0).$ Then by  generalized Taylor's theorem,
	\begin{equation*} 
		K(Q_n \varphi)-K(\varphi)-K'(\varphi)(Q_n \varphi -\varphi) = \frac{1}{2} K''(\varphi)(Q_n \varphi -\varphi)^2 + \frac{1}{3!} K^{(3)}(\varphi)(Q_n \varphi -\varphi)^3 +R(Q_n \varphi -\varphi),
	\end{equation*}
	where
	\begin{equation}  \label {Eq: 031}
		R(Q_n \varphi -\varphi)(s)=\int_0^1 \left[K^{(4)}(\varphi + \theta (Q_n \varphi -\varphi))(Q_n \varphi -\varphi)^4\right](s) \dfrac{(1-\theta)^3}{6} ~ d\theta, ~ s \in [0,1].
	\end{equation}
Now	\begin{equation*} 
		\left[K^{(4)}(\varphi + \theta (Q_n \varphi -\varphi))(Q_n \varphi -\varphi)^4\right](s) 
		= \int_0^1 \dfrac{\partial^4 \kappa}{\partial u^4}\left(s,t,\varphi(t)+ \theta (Q_n \varphi -\varphi)(t)\right)  (Q_n \varphi -\varphi)^4(t)~dt.
	\end{equation*}
	It follows that
	\begin{equation*}
		\left|\left[K^{(4)}(\varphi + \theta (Q_n \varphi -\varphi))(Q_n \varphi -\varphi)^4\right](s)\right| \le  \left( \sup_{\substack{s,t \in [0,1] \\ |u| \le \norm{\varphi}_\infty +\delta_0}} \left| \dfrac{\partial^4 \kappa}{\partial u^4}(s,t,u)\right| \right) \norm{Q_n \varphi -\varphi}_\infty^4.
	\end{equation*}
	Thus, $\displaystyle{\norm{\left[K^{(4)}(\varphi + \theta (Q_n \varphi -\varphi))(Q_n \varphi -\varphi)^4\right]}_\infty \le C_1 \norm{Q_n \varphi -\varphi}_\infty^4}$ and, therefore the equation \eqref{Eq: 031} gives	\begin{equation*} 
		\norm{ R(Q_n \varphi -\varphi)}_\infty \le \dfrac{C_1}{24} \norm{Q_n \varphi -\varphi}_\infty^4.
	\end{equation*}
	Note that
	\begin{align} \label {Eq: 032}
		K'(\varphi)(I-Q_n) [K(Q_n \varphi)-K(\varphi)-K'(\varphi)(Q_n \varphi -\varphi)] &=  \frac{1}{2}	K'(\varphi)(I-Q_n) K''(\varphi)(Q_n \varphi -\varphi)^2 \nonumber \\ &+ \frac{1}{6} K'(\varphi)(I-Q_n) K^{(3)}(\varphi)(Q_n \varphi -\varphi)^3  \nonumber \\ &+ K'(\varphi)(I-Q_n)R(Q_n \varphi -\varphi).   
	\end{align}
Now, from the Remark \ref{r01} and the equation \eqref{Eq: 07}, we obtain
	\begin{equation} \label {Eq: 033}
		\norm{K'(\varphi)(I-Q_n) K''(\varphi)(Q_n \varphi -\varphi)^2}_\infty \le \norm{K'(\varphi)(I-Q_n) K''(\varphi)} \norm{Q_n \varphi -\varphi}_{\infty}^2  
		= O\left(h^{4r+4}\right),
	\end{equation}
	\begin{align} \label {Eq: 034}
		\norm{K'(\varphi)(I-Q_n) K^{(3)}(\varphi)(Q_n \varphi -\varphi)^3}_\infty \le \norm{K'(\varphi)(I-Q_n)K^{(3)}(\varphi)} \norm{Q_n \varphi -\varphi}_{\infty}^3 = O\left(h^{6r+5}\right).
	\end{align}
	Also,
	\begin{equation*}
		\norm{K'(\varphi)(I-Q_n)R(Q_n \varphi -\varphi)}_\infty \le \norm{K'(\varphi)}\norm{I-Q_n}\norm{R(Q_n \varphi -\varphi)}_\infty \leq C_{14} \norm{R(Q_n \varphi -\varphi)}_\infty 
		=O\left( h^{8r+4} \right),
	\end{equation*}
	for some constant $C_{14}$.
	Hence, using \eqref{Eq: 033}, \eqref{Eq: 034} and the above estimate in \eqref{Eq: 032}, it follows that
	\begin{equation*} 
		\norm{K'(\varphi)(I-Q_n) \big [K(Q_n \varphi)-K(\varphi)-K'(\varphi)(Q_n \varphi -\varphi)\big]}_\infty =  O\left(h^{4r+4}\right).
	\end{equation*}
	This completes the proof.
\end{proof}

\begin{lemma} \label{l05}
	Let $r \ge 1$ and $f \in C^{2r+1}[0, 1]$, then
	\begin{eqnarray*}
		\norm{K'(\varphi) \big[K(\varphi)-K_n^{M}(\varphi)\big]}_\infty =  O\left(h^{2r+3}\right).
	\end{eqnarray*}
\end{lemma}

\begin{proof}   We have $K_n^{M}(\varphi)=Q_n K(\varphi) +K(Q_n \varphi)-Q_nK(Q_n \varphi)$. Consider 
	\begin{align*} 
		K(\varphi)-K_n^{M}(\varphi) &= (I-Q_n)(K(\varphi)-K(Q_n\varphi)) \notag \\
		&=-(I-Q_n) [K(Q_n \varphi)-K(\varphi)-K'(\varphi)(Q_n \varphi -\varphi)]  +(I-Q_n)K'(\varphi)(I-Q_n)\varphi.
	\end{align*}
	Thus, we have
		\begin{align} \label {Eq: 035}
			\norm{K'(\varphi) \big[K(\varphi)-K_n^{M}(\varphi)\big]}_\infty &\le \norm{K'(\varphi)(I-Q_n) \big[K(Q_n \varphi)-K(\varphi)-K'(\varphi)(Q_n \varphi -\varphi)\big]}_\infty \notag \\
			&\quad +\norm{K'(\varphi)(I-Q_n)K'(\varphi)(I-Q_n)\varphi}_\infty.
		\end{align}
	By Proposition \ref{p02}, we have
	\begin{equation*} 
		\norm{K'(\varphi)(I-Q_n)K'(\varphi)(I-Q_n)\varphi}_\infty =O\left( h^{2r+3} \right).
	\end{equation*}
	Hence the required result follows from the above equation, the Lemma \ref{l04} and \eqref{Eq: 035}.
	
\end{proof}


\begin{lemma} \label{l06}
	Let $r \ge 1$ and $f \in C^{2r+1}[0, 1]$, then
	\begin{eqnarray*}
		\norm{K'(\varphi) \left(\left(K_n^{M}\right)'(\varphi)-K'(\varphi)) \left( \varphi-\varphi_n^{M} \right) \right)}_\infty =  O\left(h^{2r+4}\right).
	\end{eqnarray*}
\end{lemma}

\begin{proof}   We can write
	\begin{equation} \label {Eq: 036}
		K'(\varphi) \left( (K_n^{M})'(\varphi)-K'(\varphi) \right) = -K'(\varphi) (I-Q_n) K'(\varphi)  + K'(\varphi)(I-Q_n)K'(Q_n\varphi)Q_n.
	\end{equation}
	Also,
	\begin{equation*} 
		K'(\varphi)(I-Q_n)K'(Q_n\varphi)Q_n = K'(\varphi)(I-Q_n)(K'(Q_n \varphi)-K'(\varphi))Q_n  + K'(\varphi)(I-Q_n)K'(\varphi)Q_n.
	\end{equation*}
	Now, from the equation \eqref{Eq: 03}, we get
	\begin{equation*} 
		\norm{K'(Q_n \varphi)-K'(\varphi)} \le \gamma \norm{Q_n \varphi -\varphi}_\infty.
	\end{equation*}
	Therefore
	\begin{equation*}
		\norm{K'(\varphi)(I-Q_n)(K'(Q_n \varphi)-K'(\varphi))Q_n} \leq \norm{K'(\varphi)} \left(1 + \norm{Q_n}\right) \gamma \norm{Q_n} \norm{Q_n \varphi -\varphi}_\infty = O \left( h^{2r+1} \right).
	\end{equation*}
	Using the Proposition \ref{p02} and the above estimate, we obtain
	\begin{equation*} 
		\norm{K'(\varphi)(I-Q_n)K'(Q_n\varphi)Q_n} = O\left(h^{2}\right).
	\end{equation*}
	Further, using Proposition \ref{p02} and the above estimate in equation \eqref{Eq: 036}, we have
	\begin{equation*} 
		\norm{K'(\varphi) \left( (K_n^{M})'(\varphi)-K'(\varphi) \right)} = O\left(h^2\right).
	\end{equation*}
	Hence, 
		\begin{equation*}
			\norm{K'(\varphi) \big((K_n^{M})'(\varphi)-K'(\varphi))(\varphi-\varphi_n^{M}) \big)}_\infty \le \norm{K'(\varphi) \big((K_n^{M})'(\varphi)-K'(\varphi))} \norm{\varphi_n^{M}- \varphi}_\infty,
		\end{equation*}
and the Theorem \ref{t01} give the required result.
\end{proof}
\begin{theorem} \label{t02}
	Let $K$ be a Urysohn integral operator defined as \eqref{Eq: 01} with Green's function type kernel and $f \in C^{2r + 1}[0,1]$ be a given function. Consider the equation \eqref{Eq: 02} with $\varphi$ as the unique solution in $\mathcal{B}(\varphi, \delta_0)$. Let $Q_n$ be the interpolatory operator onto the approximating space $\mathcal{X}_n$ defined by \eqref{Eq: 06}. Let $\widetilde{\varphi}_n^M \in \mathcal{B}(\varphi, \delta_0)$ be the iterated modified collocation solution defined by \eqref{Eq: 023}. Then
	\begin{eqnarray*}
		\norm{\varphi - \widetilde{\varphi}_n^M}_\infty =  \begin{cases}
			O\left( h^{4} \right), ~~ & r = 0, \\
			O\left( h^{2r+3} \right), ~~ & r \geq 1.
		\end{cases}.
	\end{eqnarray*}
\end{theorem}

\begin{proof} 
	Case I: When $r = 0$. See \cite{RPKTJN}.\\
	Case II: When $r \geq 1$. \\
	We have $\varphi =K(\varphi)+f$ and $\widetilde{\varphi}_n^M=K(\varphi_n^M)+f$, therefore
	$\widetilde{\varphi}_n^M-\varphi=K(\varphi_n^M)-K(\varphi).$
	From Grammont et al \cite{LGRPK2}, we also have
	\begin{equation} \label{Eq: 037}
		\norm{\widetilde{\varphi}_n^M- \varphi}_\infty \le \norm{K'(\varphi)(\varphi_n^M-\varphi)}_\infty + C_1 \norm{\varphi_n^M -\varphi}_\infty^2.
	\end{equation}
	Now,
		\begin{align} \label{Eq: 038}
			K'(\varphi)(\varphi_n^{M} - \varphi)  &= -(I-K'(\varphi))^{-1} K'(\varphi)\big[ K(\varphi) - K_n^M(\varphi)\big] \notag \\ 
			&\quad + (I-K'(\varphi))^{-1} K'(\varphi)\big[K_n^M(\varphi_n^{M}) -K_n^M(\varphi) - (K_n^M)'(\varphi)(\varphi_n^M -\varphi) \big] \notag \\ 
			& \quad + (I-K'(\varphi))^{-1} K'(\varphi)\big[((K_n^M)'(\varphi) - K'(\varphi))(\varphi_n^M -\varphi) \big].
		\end{align}
	Bound for the second term of the above equation is obtained using Grammont et al \cite[Lemma 3.3]{LGRPK2} as
	$$
	\norm{K_n^M(\varphi_n^{M}) -K_n^M(\varphi) - (K_n^M)'(\varphi)(\varphi_n^M -\varphi)}_\infty = O\left(\norm{\varphi_n^{M}- \varphi}_\infty^2\right) = O\left( h^{4r+4} \right).
	$$
	By Lemma \ref{l05} and Lemma \ref{l06}, we obtained the the first and third terms of \eqref{Eq: 038} are of the orders $h^{2r+3}$ and $h^{2r+4}$, respectively. Thus,
	\begin{equation*} 
		\norm{ K'(\varphi)(\varphi_n^{M}- \varphi)}_\infty = O\left( h^{2r+3} \right).
	\end{equation*}
	It follows from \eqref{Eq: 037} and the above estimate that
	\begin{eqnarray*}
		\norm{ \varphi - \widetilde{\varphi}_n^M}_\infty =  O\left( h^{2r+3} \right),
	\end{eqnarray*}
	which completes the proof.
	
\end{proof}
\begin{note} \label{n3}
	Let $K$ be a Urysohn integral operator defined as \eqref{Eq: 01} with a smooth kernel such that $\ell_{*} \in C^{4r +2}([0,1] \times [0,1])$ and $Q_n$ be the interpolatory operator onto the approximating space $\mathcal{X}_n$ defined by \eqref{Eq: 06}. If $f \in C^{4r}[0,1]$, then by Rane \cite{RANE}
	\begin{equation*}
		\norm{K'(\varphi)(I - Q_n) K'(\varphi) (I -Q_n)\varphi}_\infty = O\left(h^{4r+4} \right).
	\end{equation*}
	From Grammont et al \cite[Lemma 3.1]{LGRPK2}, we get
	\begin{equation*}
		\norm{K'(\varphi)(I-Q_n) \big [K(Q_n \varphi)-K(\varphi)-K'(\varphi)(Q_n \varphi -\varphi)\big]}_\infty =  O\left( h^{8r+4} \right).
	\end{equation*}
	Then, using the above two estimates in \eqref{Eq: 035}, we obtain
	\begin{equation} \label{Eq: 039}
		\norm{K'(\varphi) \big[K(\varphi)-K_n^{M}(\varphi)\big]}_\infty =  O\left( h^{4r + 4} \right).
	\end{equation}
	Also, from Grammont et al \cite[Lemma 3.3]{LGRPK2}, it can be shown that
	\begin{equation} \label{Eq: 040}
		\norm{K_n^M(\varphi_n^{M}) -K_n^M(\varphi) - (K_n^M)'(\varphi)(\varphi_n^M -\varphi)}_\infty = O\left( h^{8r+6} \right).
	\end{equation}
	Using the Theorem 4.2 of Kulkarni-Ganneshwar \cite{RPKGN},  it can be shown that
	\begin{equation} \label{Eq: 041}
		\norm{K'(\varphi) \left((K_n^{M})'(\varphi)-K'(\varphi) \right) \left( \varphi-\varphi_n^{M} \right)}_\infty =  O\left( h^{6r+4} \right).
	\end{equation}
	Therefore, using the equations \eqref{Eq: 039}, \eqref{Eq: 040} and \eqref{Eq: 041} in \eqref{Eq: 038} first and subsequently to \eqref{Eq: 037}, we obtain
	\begin{equation*}
		\norm{\varphi - \widetilde{\varphi}_n^M }_\infty =  O\left( h^{4r+4} \right).
	\end{equation*}
\end{note}


\section{Numerical Results}

For illustration, we consider two examples: one with a smooth kernel and the other with a non-smooth kernel. These numerical examples verify the convergence rates obtained in this article. Consider the uniform partition of $[0,1]$ as
 $\left\{ 0 =t_0 < t_1 < \cdots <t_n=1 \right\}$, where $t_j = \frac{j}{n}, \quad j=1,2, \ldots, n.$ 

\begin{example}
	Consider the following nonlinear integral equation 
	\begin{equation*}
		x  (s) - \int_0^1 \frac{1}{s + t + x(t)} ~dt  =f(s), \;\;\; 0 \leq s \leq 1,
	\end{equation*}
	where we choose $f$ such that
	$\displaystyle{\varphi(s) = \frac{1}{s+1}}$
	is the exact solution.
\end{example}	
Let $\mathcal{X}_n$ denote the space of piecewise constant functions associated with the uniform partition of $[0,1]$, and let $Q_n: L^\infty[0,1] \mapsto \mathcal{X}_n$ be the interpolatory projection defined at the midpoints, as given by \eqref{Eq: 06}. It is observed that the iterated collocation method achieves a convergence rate of $O(1/n^2)$, while the iterated modified collocation method attains a significantly higher rate of $O(1/n^4)$. We replace all the integrals appeared in the above methods by the composite 2-point Gaussian quadrature rule. In the tables presented, $\delta_C$, $\delta_S$, $\delta_M$, and $\delta_{IM}$ denote the computed orders of convergence corresponding to the collocation, iterated collocation, modified collocation, and iterated modified collocation methods, respectively.
\begin{table}[h!] 
	\centering
	\caption{\footnotesize Comparison of convergence rates for variants of collocation methods in the case of smooth kernel}
	\label{table: 01}
	\begin{tabular}{ |c| c c|c c|c c| c c| } \hline
		$n$ & $\|\varphi - \varphi_n^C \|_\infty$ & $\delta_C$ & $\|\varphi - \varphi_n^S \|_\infty$ & $\delta_S$ & $\|\varphi - \varphi_n^M \|_\infty$ & $\delta_M$ & $\|\varphi - \widetilde{\varphi}_n^M \|_\infty$ & $\delta_{IM}$ \\
		\hline
		2  & $1.93 \times 10^{-1}$ &           & $1.27 \times 10^{-2}$ &        & $6.92 \times 10^{-3}$ &           & $3.30 \times 10^{-4}$ &        \\
		4  & $1.08 \times 10^{-1}$ & $0.84$    & $3.15 \times 10^{-3}$ & $2.01$ & $1.02 \times 10^{-3}$ & $2.76$    & $2.13 \times 10^{-5}$ & $3.95$ \\
		8  & $5.73 \times 10^{-2}$ & $0.91$    & $7.86 \times 10^{-4}$ & $2.00$ & $1.40 \times 10^{-4}$ & $2.87$    & $1.34 \times 10^{-6}$ & $3.99$\\
		16 & $2.93 \times 10^{-2}$ & $0.97$    & $1.96 \times 10^{-4}$ & $2.00$ & $1.82 \times 10^{-5}$ & $2.94$    & $8.37 \times 10^{-8}$ & $4.00$\\
		32 & $1.45 \times 10^{-2}$ & $1.01$    & $4.91 \times 10^{-5}$ & $2.00$ & $2.27 \times 10^{-6}$ & $3.00$    & $5.23 \times 10^{-9}$ & $4.00$\\
		\hline
	\end{tabular}
\end{table}

As shown in Table \ref{table: 01}, under smooth kernel conditions, the computed orders of convergence for both collocation and iterated collocation solutions match well with the theoretical findings (see Note \ref{n01}). Similarly, this example confirms that the computed orders for modified collocation and iterated modified collocation solutions also match the theoretical results (see Note \ref{n02} and Note \ref{n3}), thereby validating the effectiveness of higher order collocation methods and one step iteration in achieving accelerated convergence.

\begin{example}
	Consider
	\begin{equation*}
		x  (s) - \int_0^1 \kappa (s, t) \left [ g (t, x (t) \right ] d t  = \int_0^1 \kappa (s, t) f (t) d t, \;\;\; 0 \leq s \leq 1,
	\end{equation*}
	where 
	\[\displaystyle{\kappa (s,t) = \left\{ {\begin{array}{ll}
				s ( 1 - t), & \quad 0 \leq s \leq t , \\
				(1 - s) t, & \quad  t \leq s \leq 1,
		\end{array}}\right.}	
	\quad \text{and} \quad \displaystyle{g (t, u) = \frac {1} {1 + t + u}} 
	\]
	with $f (t)$ so chosen that
	$ \displaystyle{\varphi (s) = \frac { s (1 - s)} { s + 1}}$
	is the exact solution. \\This integral equation arises from the boundary value problem (see \cite{KEA1})
	\[
	x''(s) = g(s, x(s)) + f(s), \quad 0 < s < 1, \quad x(0) = x(1) = 0.
	\]
\end{example}	
Nonlinear integral equations of this type frequently arise in the modeling of complex physical phenomena and in the reformulation of boundary value problems with nonlinear boundary conditions associated with differential equations (see \cite{KEA0, FPZ, Penn}).
Let \( \mathcal{X}_n \) denote the space of piecewise constant functions defined with respect to the given uniform partition of \([0, 1]\). Define the interpolatory projection operator \( Q_n: L^\infty[0, 1] \mapsto \mathcal{X}_n \) as in equation~\eqref{Eq: 06}, where the projection is taken at the midpoints of the subintervals. For numerical integration, we use the composite $2$-point Gauss quadrature rule corresponding to the same partition. To obtain the order convergence of $\widetilde{\varphi}_n^M$ as \(O(1/n^4) \), we apply composite 2-point Gauss quadrature rule with \( n^2 \) subintervals.
\begin{table}[h!] 
	\centering
	\caption{\footnotesize Comparison of convergence rates for variants of collocation methods in the case of Green's kernel}
	\label{table: 02}
	\begin{tabular}{ |c| c c|c c|c c| c c| } \hline
		$n$ & $\|\varphi - \varphi_n^C \|_\infty$ & $\delta_C$ & $\|\varphi - \varphi_n^S \|_\infty$ & $\delta_S$ & $\|\varphi - \varphi_n^M \|_\infty$ & $\delta_M$ & $\|\varphi - \widetilde{\varphi}_n^M \|_\infty$ & $\delta_{IM}$ \\
		\hline
		2  & $1.50 \times 10^{-1}$ &        & $1.30 \times 10^{-3}$ &        & $1.31 \times 10^{-3}$ &        & $1.31 \times 10^{-3}$ &         \\
		4  & $9.54 \times 10^{-2}$ & $0.65$ & $2.31 \times 10^{-4}$ & $2.49$ & $1.68 \times 10^{-4}$ & $2.97$ & $7.77 \times 10^{-5}$ & $4.07$  \\
		6  & $6.87 \times 10^{-2}$ & $0.81$ & $1.02 \times 10^{-4}$ & $2.01$ & $5.71 \times 10^{-5}$ & $2.66$ & $1.47 \times 10^{-5}$ & $4.10$  \\
		8  & $5.34 \times 10^{-2}$ & $0.88$ & $5.81 \times 10^{-5}$ & $1.95$ & $2.62 \times 10^{-5}$ & $2.71$ & $4.76 \times 10^{-6}$ & $3.92$ \\
		10 & $4.35 \times 10^{-2}$ & $0.92$ & $3.67 \times 10^{-5}$ & $2.07$ & $1.37 \times 10^{-5}$ & $2.90$ & $1.87 \times 10^{-6}$ & $4.19$ \\
		12 & $3.66 \times 10^{-2}$ & $0.95$ & $2.59 \times 10^{-5}$ & $1.90$ & $8.22 \times 10^{-6}$ & $2.80$ & $9.52 \times 10^{-7}$ & $3.69$\\
		\hline
	\end{tabular}
\end{table}
\\
In Table \ref{table: 02}, a numerical example with Green's function-type kernel demonstrates that iterations improve the orders of convergence for both collocation and modified collocation solutions. In the collocation method improving from approximately first order to second order convergence upon iteration, and the modified collocation solution enhancing from third to nearly fourth order accuracy. This highlights the robustness of the iterated approaches, particularly the iterated modified collocation method, which maintains high convergence rates even in the presence of reduced kernel smoothness due to discontinuities along the diagonal.

\section{Conclusion}
We study the collocation method and its variants for the approximate solutions of Urysohn integral equations with smooth and Green’s function-type kernels. The projection used is an interpolatory projection at $2r+1$ collocation points defined on each subinterval with respect to a partition of $[0,1]$. Previous studies have shown that selecting Gauss points (i.e., the zeros of Legendre polynomials) as collocation points enhances the convergence rates of iterative methods. Usually, if the collocation points are not Gauss points (or not zeros of special functions), the solutions $\varphi_n^C$ (collocation solution), $\varphi_n^S$ (the iterated collocation solution), $\varphi_{n}^{m}$ (modified collocation solution) and $\widetilde{\varphi}_n^M$ (iterated modified collocation solution) exhibit the same orders of convergence, regardless of whether the kernel is smooth or of Green's function type.

For example, consider $\mathcal{X} = L^\infty [0, 1]$ and $\mathcal{X}_n$ be the space of piecewise constant functions on the uniform partition of $[0,1]$, given by $\{0 = t_0 < t_1 < \cdots < t_n = 1\}$, where $t_j = \frac{j}{n}$ for $j = 0,1, \ldots, n$. Let $Q_n : C[0, 1] \mapsto \mathcal{X}_n$ be the interpolatory map at collocation the points $ \tau_j = \frac{3j-2}{3n}$,  for $j = 1, \ldots, n$. Using the Hahn-Banach extension theorem, the domain of definition of $Q_n$ can be extended to $L^\infty [0, 1]$ and so $Q_n$ is a projection. Consider a nonlinear integral equation as
$ x - K(x) = f$, where $f(s) = s$, $s\in [0, 1]$, and 
\[
K(x)(s) = \left( \int_0^1 x(t) ~ dt + 1 \right) s, \quad s \in [0, 1].
\] 
We have $\varphi(s) = 4s$ for $s \in [0, 1]$ as the exact solution of this equation. It is evident that the operator $K$ is compact and Fr\'echet differentiable, with 
\[
\left(K'(\varphi)v\right)(s) = \left( \int_0^1 v(t) ~ dt \right) s, \quad \text{ for all } v \in \mathcal{X}, \;  s \in [0, 1].
\]
Note that $1$ is not an eigenvalue of $K'(\varphi)$.
Observe that
\[
(I - Q_n) \varphi(s)= 4s - \frac{4(3j-2)}{3n}, \quad s \in  [t_{j-1}, t_{j} ],
\]
and 
\[ \int_0^1 (I -Q_n) \varphi(t) ~dt =\sum_{j=1}^{n} ~ \int_{t_{j-1}}^{t_j} \left(4t -\frac{4(3j-2)}{3n} \right) ~ dt = \frac{2}{3n}.
\]
We get
\[
\norm{K'(\varphi)(I -Q_n) \varphi }_\infty =  \max_{s\in [0,1]} \left| \left( \int_0^1 (I -Q_n) \varphi(t) ~dt \right) s \right| = \frac{2}{3n},
\]
and
\[
\norm{(I - Q_n ) \varphi}_\infty = \max_{1 \le j \le n} \left\{ \max_{s \in  [t_{j-1}, t_{j} ]} |(I - Q_n ) \varphi(s) | \right\} \le \frac{8}{3n}.
\]
In general, the iterated collocation solution ia superconvergent (see \cite{KEA0}) if and only if 
\[
\frac{\norm{K'(\varphi)(I -Q_n) \varphi }_\infty }{ \norm{(I - Q_n ) \varphi}_\infty} \to 0 \; \text{as} \; n \to \infty.
\]
In this case, the above condition is not satisfied, because
\[
\frac{\norm{K'(\varphi)(I -Q_n) \varphi }_\infty }{ \norm{(I - Q_n ) \varphi}_\infty} \ge \frac{1}{4}.
\] 

In this paper, we consider a sequence of interpolatory projections at $2r + 1$ equidistant collocation points (not necessarily Gauss points) onto the space of piecewise polynomials of even degree $\le 2r$ with respect to a uniform partition $\{0 = t_0 < t_1 < \cdots < t_n = 1\}$, where $t_j = \frac{j}{n} = j h$. The key findings are as follows: \\For smooth kernels, the orders of convergence for the collocation, iterated collocation, modified collocation, and iterated modified collocation methods are $h^{2r+1}$, $h^{2r+2}$, $h^{4r+3}$ and $h^{4r+4}$, respectively. In contrast, for Green’s function-type kernels, the respective convergence orders are $h^{2r+1}$, $h^{2r+2}$, $h^{2r+2}$ and $h^{2r+3}$, respectively. These findings clearly demonstrate that the iterated modified collocation method achieves the fastest convergence, as the sequence $\{ \widetilde{\varphi}_n^M \}$ approaches the exact solution $\varphi$ more rapidly than the other approximations, verified by the numerical results.
In future, we want to employ these methods to integral equations with singular or weakly singular kernels, developing asymptotic series expansions for iterated solutions, and generalizing the theoretical framework and numerical schemes to higher-dimensional problems such as boundary integral equations. 

	\section*{Declarations}
	The authors have no conflict of interest to declare.
	


	
	
\end{document}